\documentclass[reqno]{amsart}

\usepackage[pdftex]{graphicx}
\usepackage[pagewise]{lineno}

\begin{document}
    \title[\hfilneg  \hfil Bresse-Timoshenko type systems with thermodiffusion effects]
    {Bresse-Timoshenko type systems with thermodiffusion effects: Well-possedness, stability and numerical results}

 \author[M. Elhindi et {\em al.}\hfil  \hfilneg]    {M. Elhindi, Kh. Zennir, D. Ouchenane, A. Choucha and T. EL Arwadi}

\address{Mohammad Elhindi\newline  Department of Mathematics and Computer	Science, Faculty of Science, Beirut Arab University, Beirut, Lebanon  } \email{myh223@student.bau.edu.lb}

\address{Khaled zennir \newline	Department of Mathematics, College of Sciences and Arts, Qassim University, Ar-Rass, Saudi Arabia.\newline 	Laboratoire de Math\'ematiques Appliqu\'ees et de Mod\'elisation, 	Universit\'e 8 Mai 1945 Guelma. B.P. 401 Guelma 24000 Alg\'erie}\email{khaledzennir4@gmail.com}

\address{Djamel Ouchenane \newline Laboratory of pure and applied mathematic, Laghouat University, Algeria} \email{ouchenanedjamel@gmail.com}

\address{Abdelbaki Choucha \newline 	Department of Mathematics, Faculty of Exact Sciences, University of El Oued, B.P. 789, El Oued 39000, Algeria} \email{abdelbaki.choucha@gmail.com}

\address{Toufic EL Arwadi \newline	Department of Mathematics and Computer 	Science, Faculty of Science, Beirut Arab University, Beirut, Lebanon}\email{t.elarwadi@bau.edu.lb}


    \begin{abstract}
Bresse-Timoshenko beam model with thermal, mass diffusion and theormoelastic effects is studied. We state and prove the well-posedness of problem. The global existence and uniqueness of the solution is proved by using the
classical Faedo-Galerkin approximations along with two a priori estimates. We prove an exponential stability estimate for problem under an unusual assumption, and by using a multiplier technique in two different cases, with frictional damping in the angular rotation and with frictional damping in the vertical displacement. In numerical parts, we first obtained a numerical scheme for problem by $P_1$-finite element method for space discretization and implicit Euler scheme for time discretization. Then, we showed that the discrete energy decays, later a priori error estimates are established. Finally , some numerical simulations are presented.
    \end{abstract}
    \maketitle
    \numberwithin{equation}{section}
    \newtheorem{theorem}{Theorem}[section]
    \newtheorem{lemma}[theorem]{Lemma}
    \newtheorem{definition}[theorem]{Definition}
    \newtheorem{remark}[theorem]{Remark}
    \allowdisplaybreaks

    \section{Introduction and position of problem }

    In engineering practice, when solving problems of the dynamics of composite mechanical structures, which are various kinds of connections,  questions arise on determining the characteristics of natural vibrations of such coupled systems. Note that problems related to the category of non-classical problems of mathematical physics, when we talk about the combination of elements, the behavior of which is described by equations of different type. This causes certain difficulties in solving them, therefore, in practice, models of real structures are used, simplified by introducing additional hypotheses and assumptions into consideration.\\
    Timoshenko \cite{Timoshenko1921}, was the first who introduced the system of the form
    \begin{equation}
    \left\{
    \begin{array}{l}
    \rho  \varphi_{tt}-\kappa( \varphi_{x}+\psi) _{x}=0\\
    I_{\rho}\varphi _{tt}- (EI\psi _{x})_x+\kappa(\varphi_{x}+\psi)  =0.
    \end{array}%
    \right. \label{Kirchh}
    \end{equation}
    Here, $\varphi$ is the transverse displacement of the beam and $\psi$ is the rotation angle of the filament of the beam. The coefficients $\rho, I_\rho, E, I$ and $\kappa$ are respectively the density (the mass
    per unit length), the polar moment of inertia of a cross section, Young's modulus of elasticity, the moment of inertia of a cross section, and the shear modulus. This kind of systems have been studied by a number of researchers and various damping mechanisms have been used to stabilize the vibrations. (\cite{Metrikine2001,Suiker1998})\\
    The Bresse system or the curved beam \cite{Bresse}, is modeled by the system
    \begin{equation}
    \left\{
    \begin{array}{l}
    \rho_{1} \varphi_{tt}-\kappa( \varphi_{x}+lw+\psi) _{x}-l\kappa_0( w_{x}-l\varphi) =0\\
    \rho_{2}\psi _{tt}-b\psi_{xx}+\kappa( \varphi_{x}+lw+\psi)  =0 \\
    \rho_{1}w_{tt}-\kappa_0( w_{x}-l\varphi)_x+l\kappa( \varphi_{x}+lw+\psi)=0.
    \end{array}%
    \right. \label{Bresse}
    \end{equation}
    The terms $\kappa_0( w_{x}-l\varphi), \kappa( \varphi_{x}+lw+\psi)$ and $b\psi_{x}$ denote the axial force, the shear force and the bending moment. The functions $\varphi, \psi$ and $w$ represent, respectively, the transverse displacement of a curved beam, the rotation angle of the
    filament and the longitudinal displacement. We denote by $\kappa_0=EH, \kappa=GH, b=EI$ and $\rho_{1}, \rho_{2}, l, G, E, H$ are positive constants characterizing physical properties of the beam and the filament. In addition, $l=1/R$, where $R$ is the radius of curvature. (\cite{Almeida2016,Arwadi2019})\\
    The coupled system from where one gets the Bresse-Timoshenko comes from Elishakoff \cite{Elishakoff} by combining d'Alembert's principle for dynamic equilibrium from Timoshenko hypothesis, resulting the coupled system
    \begin{equation}
    \left\{
    \begin{array}{l}
    \rho_{1} \varphi_{tt}-\kappa( \varphi_{x}+\psi) _{x}=0\\
    -\rho_{2}\varphi _{ttx}-b \psi _{xx}+\kappa(\varphi_{x}+\psi) =0.
    \end{array}%
    \right. \label{Bresse-Timoshenko}
    \end{equation}
    In the classical theory of thermoelasticity, the behavior of an elastic heat body can be described by a coupled system of hyperbolic-parabolic type, where the classical Fourier model of heat conduction is used, one most famous among them is the Cattaneo's law, which is unable to account for some physical properties and it cannot answer all questions, its uses are limited, this let us think to couple the fields of strain, temperature, and mass diffusion according to the Gurtin-Pinkin model.
   The stabilisation of the Bresse-Timoshenko model is studied only by few authors. We review the work in \cite{Aouadi2019}, a new Timoshenko beam model with thermal and mass diffusion effects according to the Gurtin-Pinkin model is proposed. The author proved global well-posedness of system by using the semigroup theory and also the quasistability. Despite the fact that a sufficient number of works have been devoted to the study of natural vibrations of a Breese-Timoshenko beam,  the problem of determining qualitative properties with thermal, mass diffusion and theormoelastic effects remains unsolved. \cite{Almeida2018,Almeida2019,Almeida20192,Choucha2020,Feng2020,Ramos2020}\\
    In \cite{Arwadi2019}, the authors studied stability of thermoviscoelastic Bresse beam system. The exponential decay of energy is proved and implicit Euler type scheme based on finite differences in time and finite elements in spaces is introduced to show that the discrete energy decreases in time and the author obtained an error estimates.\\
    In \cite{Feng2020}, Feng and {\em al.}, considered a Bresse-Timoshenko type system
    with time-dependent delay terms in $ ]0, L[ \times ]0, \infty[$
    \begin{equation}
    \left\{
    \begin{array}{l}
    \rho_{1} y_{tt}-\kappa( y_{x}+\psi) _{x}=0\\
    -\rho_{2}y _{ttx}-b \psi _{xx}+\kappa(y_{x}+\psi)+\mu_1 \psi_t +\mu_2\psi_t(t-\tau(t)) =0,\\
    \end{array}%
    \right. \label{0}
    \end{equation}
    and
    \begin{equation}
    \left\{
    \begin{array}{l}
    \rho_{1} y_{tt}-\kappa( y_{x}+\psi) _{x}+\mu_1 y_t +\mu_2y_t(t-\tau(t))=0 \\
    -\rho_{2}y _{ttx}-b \psi _{xx}+\kappa(y_{x}+\psi) =0.\\
    \end{array}%
    \right. \label{2}
    \end{equation}
    In both systems (\ref{0}) and (\ref{2}), the authors used an appropriate
    Lyapunov functional to prove an exponential decay results,  regardless of any relationship between wave propagation velocities.(See \cite{Almeida2017,Almeida2018,Almeida2019,Ramos2020}). The present article is a logical continuation of works \cite{Aouadi2019,Choucha2020,Feng2020}.\\
   We  introduce a new Bresse-Timoshenko beam model with thermal, mass diffusion and theormoelastic effects. The beam is modeled by the following system
    \begin{equation}
    \left\{
    \begin{array}{l}
    \rho_{1} \varphi_{tt}-\kappa( \varphi_{x}+\psi) _{x}=0\\
    -\rho_{2}\varphi _{ttx}-b \psi _{xx}+\kappa(\varphi_{x}+\psi)-\gamma\theta_{x}-\beta  C_{x}  =0\\
    \rho_{3}\theta_{t}+\varpi C_{t}-\kappa\theta_{xx}-\gamma\psi_{tx}=0\\
     C_{t}-h(\beta\psi_{x}+\rho C-\varpi\theta)_{xx}=0,
    \end{array}%
    \right. \label{sys1.11}
    \end{equation}
    where
    $$(x, t)\in(0, L)\times(0, \infty),$$
    where $L$ represents the distance between the ends of
    the center line of beam. The function $C$ denote the concentration of the diffusive material in the elastic body. Here $h > 0$ is the diffusion coefficient, $\varpi$ is a measure of the thermo-diffusion effect. In order to simplify the system we use the following relation between chemical potential $P$ and the
    concentration of the diffusion material $C$
    $$C=\frac{1}{\varrho}(P-\beta\psi_{x}+\varpi\theta).$$
    Here $\varrho$ is a measure of the diffusive effect, we put
    $$\alpha=b-\frac{\beta^{2}}{\varrho}, \hspace{0.2cm}\xi_{1}=\gamma+\frac{\beta\varpi}{\varrho}, \hspace{0.2cm}\xi_{2}=\frac{\beta}{\varrho}, \hspace{0.2cm}c=\rho_{3}+\frac{\varpi}{\varrho}, \hspace{0.2cm}r=\frac{1}{\varrho}.$$
    Substitute in (\ref{sys1.11}), the problem becomes
    \begin{equation}
    \left\{
    \begin{array}{l}
    \rho_{1} \varphi_{tt}-\kappa( \varphi_{x}+\psi) _{x}=0\\
    -\rho_{2}\varphi _{ttx}-\alpha \psi _{xx}+\kappa(\varphi_{x}+\psi)-\xi_{1}\theta_{x}-\xi_{2} P_{x} =0\\
    c\theta_{t}+d P_{t}- \kappa \theta_{xx}-\xi_{1}\psi_{tx}=0\\
    d\theta_{t}+rP_{t}-hP_{xx}-\xi_{2}\psi_{tx}=0.
    \end{array}%
    \right. \label{sys1.15}
    \end{equation}
   The aim of the paper is to study system (\ref{sys1.15}) with the following initial conditions
    \begin{equation}
    \left\{
    \begin{array}{l}
    \varphi\left( x, 0\right) =\varphi_{0}\left( x\right), \varphi_{t}\left( x, 0\right)
    =\varphi_{1}\left( x\right), \varphi_{tt}\left( x, 0\right)
    =\varphi_{2}\left( x\right) \\
   \psi \left( x, 0\right) =\psi _{0}\left( x\right), \psi_{t}\left( x, 0\right)
   =\psi_{1}\left( x\right)\\
    \theta(x, 0)=\theta_{0}(x),  P \left( x, 0\right) =P _{0}\left( x \right),
    x\in \left( 0, L\right),
    \end{array}
    \right. \label{con1.1}
    \end{equation}
    where $\varphi_{0}, \varphi_{1}, \psi_{0}, \psi_{1}, \theta_{0}, P_{0}$ are given
    functions,
    and the Dirichlet boundary conditions
    \begin{equation}
   \varphi\left( x, t\right) =\psi\left( x, t\right) =\theta\left( x, t\right) =P(x, t)=0, \ x=0, L, t>0.
    \label{sys1.14}
    \end{equation}
  For $cr-d^2>0$, we assume that the symmetric matrix
       \begin{equation}
    \Lambda=
    \left(
    \begin{array}{l}
    c\hspace{0.2cm} d\\
    d\hspace{0.2cm}r
    \end{array}
    \right),
    \label{eq9}
    \end{equation}
    is positive definite, and thus for all $\theta, P$
    \begin{equation}
    rP^{2}+c\theta^{2}+2dP\theta>0.
    \label{con}
    \end{equation}
In the present paper, a new  minimal conditions on dissipation and the relationship between the weights of system terms are used to show the global existence of solution by well known Faedo-Galerkin method combined with some estimates. By imposing a new appropriate conditions, which seems not be used in the literature. With the help of some special results, we obtained an unusual decay rate results using some properties of multiplier technique, extending some earlier results known in the existing literature. The main results in this manuscript are the following. Theorem \ref{the1} for the global existence of solution and Theorem \ref{the2}, Theorem \ref{the3} for the exponential decay rate for problem $(\ref{sys1.15})-(\ref{sys1.14})$ under the assumption (\ref{eq9}), (\ref{con}) with both cases, frictional damping in the angular rotation and frictional damping in the angular rotation. We obtained a numerical scheme for the problem by $P_1$-finite element method for space discretization and implicit Euler scheme for time discretization. Then, we showed that the discrete energy decays in Theorem \ref{Th4}. The error of the method where we propose an outline for the proof is studied. Finally, some numerical simulations are obtained using mathematica software.
   \section{Global well-posedness}
   We are now ready to state and prove the global well-posedness of problem (\ref{sys1.15})-\ref{sys1.14}).\\
   The dissipative nature of our system comes from the definition of the energy functional
   \begin{eqnarray}
  2 \mathcal{E}(t)&=& \rho_{1} \int_{0}^{L}\varphi_{t}^{2}dx +\frac{\rho_{1}\rho_{2}}{\kappa}\int_{0}^{L}\varphi_{tt}^{2}dx+\rho_{2} \int_{0}^{L}\varphi_{tx}^{2}dx\notag\\
   &&+\kappa\int_{0}^{L}(\varphi_{x}+\psi)^{2}dx
   + \alpha\int_{0}^{L}\psi_{x}^{2}dx\notag\\
   &&+\Big(c-\frac{d^2}{r}\Big)\int_{0}^{L}\theta^{2}dx+\int_{0}^{L}|r^{1/2}P 
   + dr^{-1/2}\theta |^2dx.
   \label{energy}
   \end{eqnarray}
  Multiplying the equations of (\ref{sys1.15}) by $\varphi _{t}, \psi_{t}, \theta, P$ respectively, using integration by parts, and (\ref{sys1.14}), we get
  \begin{equation}
  \left\{
  \begin{array}{l}
  \dfrac{\rho _{1}}{2}\dfrac{d}{dt}\displaystyle\int_{0}^{L}\varphi
  _{t}^{2}dx+\kappa \displaystyle\int_{0}^{L}(\varphi _{x}+\psi )\varphi
  _{tx}dx=0 \\
  \rho _{2}\displaystyle\int_{0}^{L}\varphi _{tt}\psi _{tx}dx+\frac{\alpha }{2
  }\frac{d}{dt}\displaystyle\int_{0}^{L}\psi _{x}^{2}dx+\kappa
  \int_{0}^{L}(\varphi _{x}+\psi )\psi _{t}dx \\
  \hspace{1cm}+\xi _{1}\int_{0}^{L}\theta\psi _{xt}dx+\xi _{2}\int_{0}^{L}P\psi _{xt}dx=0 \\
  \dfrac{c}{2}\dfrac{d}{dt}\displaystyle\int_{0}^{L}\theta ^{2}dx+d%
  \displaystyle\int_{0}^{L}P_{t}\theta dx+\kappa \displaystyle%
  \int_{0}^{L}\theta _{x}^{2}dx+\xi _{1}\displaystyle\int_{0}^{L}\psi
  _{t}\theta_x dx=0 \\
  \dfrac{r}{2}\dfrac{d}{dt}\displaystyle\int_{0}^{L}P^{2}dx+d\displaystyle%
  \int_{0}^{L}\theta _{t}Pdx+h\displaystyle\int_{0}^{L}P_{x}^{2}dx+\xi _{2}%
  \displaystyle\int_{0}^{L}\psi _{t}P_xdx=0.
  \end{array}%
  \right.  \label{sys1.16.0}
  \end{equation}%
  Taking the derivative of $(\ref{sys1.15})_{1}$, we get
  \begin{eqnarray}
  \rho _{1}\left( \varphi _{tt}\right) _{t}-\kappa (\varphi _{x}+\psi )_{xt}
  =0,  \label{KhaledZennir}
  \end{eqnarray}
  then
  \begin{eqnarray}
  \psi _{xt} &=&\frac{\rho _{1}\left( \varphi _{tt}\right) _{t}}{%
  	\kappa }-\frac{\left( \varphi _{x}\right) _{xt}}{\kappa }.  \notag
  \end{eqnarray}%
  Now, substituting (\ref{KhaledZennir}) in $(\ref{sys1.15})_{2},$ using
  integration by parts and summing, then we obtain that $\mathcal{E}$ is decreasing and given by
   \begin{eqnarray}
   \mathcal{E}(t)-\mathcal{E}(0) =-\kappa\int_{0}^{t}\int_{0}^{L}\theta_{x}^{2}(s)dxds
   -h\int_{0}^{t}\int_{0}^{L}P_{x}^{2}(s)dxds.\label{e'}
   \end{eqnarray}
We introduce the following Hilbert spaces
 \begin{eqnarray}
\mathcal{H}=H^{1}_{0}(0, L)\times H^{1}_{0}(0, L)\times L^{2}(0, L) \times H^{1}(0, L)\times L_*^{2}(0, L),
\end{eqnarray}
where 
$$
L_*^{2}(0, L)=\Big\{u\in L^{2}(0, L): \int_{0}^{L}udx=0\Big\}.
$$
Global well-posedness is given in the following.
\begin{theorem}\label{the1}
Assume that (\ref{eq9}), (\ref{con}) hold.
If the initial data $(\varphi_{0}, \varphi_{1}, \varphi_{2}, \psi_{0}, \psi_{1}) \in  \mathcal{H},$
 and $\theta_{0}, P_{0} \in L^{2}(0, L)\times L^{2}(0, L)$.
Then problem (\ref{sys1.15})-(\ref{sys1.14}) has a unique weak solution such that
\begin{equation}
\varphi, \varphi_t \in L^\infty(\mathbb{R}_+, H^1_0(0, L)),  \nonumber
\end{equation} 
\begin{equation}
\psi \in L^\infty(\mathbb{R}_+, H^1(0, L)),  \nonumber
\end{equation}
\begin{equation}
\psi_t \in L^\infty (\mathbb{R}_+, L_*^2(0,L)),  \nonumber
\end{equation}
\begin{equation}
\varphi_{tt}, \theta, P \in C(\mathbb{R}_+, L^2(0, L)). \nonumber
\end{equation}
   In addition,the unique solution $(\varphi, \varphi_{t}, \varphi_{tt}, \psi, \psi_t, \theta, P)$ depends
continuously on the initial data in $\mathcal{H}\times
L^{2}(0, L)\times L^{2}(0, L)$.
    \end{theorem}
\begin{proof}
	We will prove global existence and uniqueness of solution for problem (\ref{sys1.15})-\ref{sys1.14}) by using the
    classical Faedo-Galerkin approximations along with two a
    priori estimates. For more detail, we refer the reader to see \cite{Andrade,Jorge,Feng4}.\\
   {\bf Approximate solutions:} Let $\{u_{j}\}, \{v_{j}\}, \{\theta_{j}\}, \{P_{j}\}, 1\leq j\leq n$ be the Galerkin basis, for every $n\geq 1$, let
   \begin{eqnarray*}
    W_{n}=span\{u_{1}, u_{2},...., u_{n}\}\\
   K_{n}=span\{v_{1}, v_{2},...., v_{n}\}\\
   \Theta_{n}=span\{\theta_{1}, \theta_{2},...., \theta_{n}\}\\
   \Gamma_{n}=span\{P_{1}, P_{2},...., P_{n}\}.
\end{eqnarray*}
	Given initial data $\varphi_{0}, \psi_{0}\in H^{1}_{0}(0, L)$, $\varphi_{1}, \varphi_{2}, \psi_{1}, \theta_{0}, P_{0}\in L^{2}(0, L)$, we seek functions $g_{jn}, \zeta_{jn}, f_{jn}, k_{jn} \in C^2[0, T]$, such that the approximations
   \begin{eqnarray}
   \left\{
   \begin{array}{ll}
   \varphi_{n}(t)=\sum\limits_{j=1}^{n}g_{jn}(t)u_{j}(x)\\
   \psi_{n}(t)=\sum\limits_{j=1}^{n} \zeta_{jn} (t)v_{j}(x)\\
   \theta_{n}(t)=\sum\limits_{j=1}^{n}f_{jn}(t)\theta_{j}(x)\\
   P_{n}(t)=\sum\limits_{j=1}^{n}k_{jn}(t)P_{j}(x),
   \end{array}
   \right. \label{2.13}
   \end{eqnarray}
    hold, which solve the following approximate problem
    \begin{eqnarray}
    \left\{
    \begin{array}{ll}
    \rho_{1} (\varphi_{ntt},  u_{j})+\kappa((\varphi_{nx}+\psi_{n}),  u_{jx})=0,  \\
    \alpha (\psi _{nx}, v_{jx})+\rho_{2}(\varphi _{ntt},  v_{jx} )+\kappa((\varphi_{nx}+\psi_{n}), v_{j})\\ \ \ \ \ \ \ \ \ +\xi_{1} (\theta _{n}, v_{jx})+\xi_{2}(P _{n}, v_{jx}) =0 \\
    c(\theta_{nt}, \theta_{j})+d(P_{nt}, \theta_{j})+ \kappa (\theta_{nx}, \theta_{jx})+\xi_{1}(\psi_{nt}, \theta_{jx})=0 \\
    d(\theta_{nt}, P_{j})+r(P_{nt}, P_{j})+ h (P_{nx}, P_{jx})+\xi_{2}(\psi_{nt}, P_{jx})=0,
    \end{array}
    \right. \label{q1.1}
    \end{eqnarray}
    with initial conditions
    \begin{eqnarray}
    &&\varphi_{n}(0)=\varphi^{n}_{0}, \varphi_{nt}(0)=\varphi^{n}_{1}, \varphi_{ntt}(0)=\varphi^{n}_{2}\notag\\
    &&\psi_{n}(0)=\psi^{n}_{0}, \psi_{nt}(0)=\psi^{n}_{1},\notag\\
    &&\theta_{n}(0)=\theta^{n}_{0}, P_{n}(0)=P^{n}_{0}.\label{q1.2}
    \end{eqnarray}
 We choose $\varphi^{n}_{0}, \varphi^{n}_{1}, \varphi^{n}_{2} \in [u_1, u_2, \dots, u_n]$, $\psi^{n}_{0}, \psi^{n}_{1} \in [v_1, v_2, \dots, v_n]$, $\theta^{n}_{0} \in [\theta_1, \theta_2, \dots, \theta_n]$ and $P^{n}_{0} \in [P_1, P_2, \dots, P_n]$ such that
    \begin{eqnarray}
    &&\varphi^{n}_{0}=\sum_{j=1}^{n}(\varphi_0, u_j)u_j\to \varphi_{0} \hspace{0.2cm} in \hspace{0.2cm} H^{1}_{0}(0, L)\notag\\
    &&\varphi^{n}_{1}=\sum_{j=1}^{n}(\varphi_1, u_j)u_j\to \varphi_{1} \hspace{0.2cm} in \hspace{0.2cm} L^{2}(0, L)\notag\\
    &&\varphi^{n}_{2}=\sum_{j=1}^{n}(\varphi_2, u_j)u_j\to \varphi_{2} \hspace{0.2cm} in \hspace{0.2cm} L^{2}(0, L)\notag\\
    &&\psi^{n}_{0}=\sum_{j=1}^{n}(\psi_0, v_j)v_j\to \psi_{0} \hspace{0.2cm} in \hspace{0.2cm} H^{1}_{0}(0, L)\notag\\
    &&\psi^{n}_{1}=\sum_{j=1}^{n}(\psi_1, v_j)v_j\to \psi_{1} \hspace{0.2cm} in \hspace{0.2cm} L^{2}(0, L)\notag\\
    &&\theta^{n}_{0}=\sum_{j=1}^{n}(\theta_0, \theta_j)\theta_j\to \theta_{0} \hspace{0.2cm} in \hspace{0.2cm} L^{2}(0, L)\notag\\
    &&P^{n}_{0}=\sum_{j=1}^{n}(P_0, P_j)P_j\to P_{0} \hspace{0.2cm} in \hspace{0.2cm} L^{2}(0, L).\notag
    \end{eqnarray}
    By using the Caratheodory theorem for standard ordinary differential equations theory, the problem (\ref{q1.1})-(\ref{q1.2}) has a solution $(g_{jn}, \zeta_{jn}, f_{jn}, k_{jn})\in (H^3[0, T])^4$ and by using the embedding $H^m[0, T] \to C^{m-1}[0, T]$, we deduce that the solution $(g_{jn}, \zeta_{jn}, f_{jn}, k_{jn})\in (C^2[0, T])^4$. In turn, this gives a unique $(\varphi_n, \psi_n, \theta_n, P_n)$ defined by (\ref{2.13}) and satisfying (\ref{q1.1}).\\
    {\bf A priori estimates:} The following estimates prove that the functional energy defined in (\ref{q1.8}) related to the problem (\ref{sys1.15}) is bounded and will give the local solution being extended to $[0, T]$, for any given $T>0$.
    \begin{enumerate}
    	\item {\bf First a priori estimate:} Now multiplying, respectively, (\ref{q1.1})$_1$, (\ref{q1.1})$_2$, (\ref{q1.1})$_3$ and (\ref{q1.1})$_4$ by $g'_{jn}, \zeta'_{jn}, f'_{jn}$ and  $k'_{jn}$. By the fact that
    \begin{eqnarray*}
    \kappa\int_{0}^{L}\varphi_{tt}\psi_{tx}dx&=&\rho_{1} \int_{0}^{L}\varphi_{ttt}\varphi_{tt}dx+\kappa\int_{0}^{L}\varphi_{txx}\varphi_{tt}dx,
    \end{eqnarray*}
    we get
    \begin{eqnarray}
    &&\frac{d}{dt}\frac{1}{2}\bigg[\rho_{1} \int_{0}^{L}\varphi_{nt}^{2}dx +\frac{\rho_{1}\rho_{2}}{\kappa}\int_{0}^{L}\varphi_{ntt}^{2}dx+\rho_{2} \int_{0}^{L}\varphi_{ntx}^{2}dx+\kappa\int_{0}^{L}(\varphi_{nx}+\psi_{n})^{2}dx\notag\\
    &&
    \quad+ \alpha\int_{0}^{L}\psi_{nx}^{2}dx+c\int_{0}^{L}\theta_{n}^{2}dx+r\int_{0}^{L}P_{n}^{2}dx
    +2d\int_{0}^{L}\theta_{n}P_{n}dx\bigg]\notag\\
    && +\kappa\int_{0}^{L}\theta_{nx}^{2}dx+h\int_{0}^{L}P_{nx}^{2}dx.\label{q1.4}
    \end{eqnarray}
        Now integrating (\ref{q1.4}), we obtain
         \begin{eqnarray}
    \mathcal{E}_{n}(t)&+&\kappa\int_{0}^{t}\int_{0}^{L}\theta_{nx}^{2}(s)dxds
    +h\int_{0}^{t}\int_{0}^{L}P_{nx}^{2}(s)dxds
    =\mathcal{E}_{n}(0), \label{q1.7}
    \end{eqnarray}
    with
    \begin{eqnarray}
    \mathcal{E}_{n}(t)&=&\frac{1}{2}\bigg[\rho_{1} \int_{0}^{L}\varphi_{nt}^{2}dx +\frac{\rho_{1}\rho_{2}}{\kappa}\int_{0}^{L}\varphi_{ntt}^{2}dx+\rho_{2} \int_{0}^{L}\varphi_{ntx}^{2}dx\notag\\
    &&+\kappa\int_{0}^{L}(\varphi_{nx}+\psi_{n})^{2}dx
    + \alpha\int_{0}^{L}\psi_{nx}^{2}dx\notag\\
    &&+c\int_{0}^{L}\theta_{n}^{2}dx+r\int_{0}^{L}P_{n}^{2}dx
    +2d\int_{0}^{L}\theta_{n}P_{n}dx\bigg].
    \label{q1.8}
    \end{eqnarray}
  We get
    \begin{eqnarray}
    \mathcal{E}_{n}(t)\leq \mathcal{E}_{n}(0).\label{q1.11}
    \end{eqnarray}
    Then, in both cases, we infer that there exists a positive constant
$C$ independent on $n$ such that
    \begin{eqnarray}
    \mathcal{E}_{n}(t)\leq C, \hspace{0.3cm}t\geq 0.\label{q1.14}
    \end{eqnarray}
    It follows from (\ref{eq9}) and (\ref{q1.14}) that
    \begin{eqnarray}
    \int_{0}^{L}\varphi_{nt}^{2}dx +\int_{0}^{L}\varphi_{ntt}^{2}dx+\rho_{2} \int_{0}^{L}\varphi_{ntx}^{2}dx+\int_{0}^{L}(\varphi_{nx}
    +\psi_{n})^{2}dx\notag\\
    + \int_{0}^{L}\psi_{nx}^{2}dx+c\int_{0}^{L}\theta_{n}^{2}dx+r\int_{0}^{L}P_{n}^{2}dx
    +2d\int_{0}^{L}\theta_{n}P_{n}dx
     \leq C.\label{q1.15}
    \end{eqnarray}
    Thus we can obtain $t_{n}=T$, for all $T>0$.\\
    The insufficient regularity due to the presence of coupled system of hyperbolic/parabolic equations, we must derive second a priori estimat to prove in the next a prior estimates that, the family of approximations defined in (\ref{2.13}) is compact in the strong topology and by using compactness of the embedding (without mention) and using Aubin-Lions Lemma \cite{Lions}, our conclusion holds with an appropriate regularity.
\item {\bf The second a priori estimate:} Differentiating equation \eqref{q1.1}$_1$ and multiplying by $\varphi_{ntt}$ and then
integrating the result over (0, L), we have
\begin{eqnarray}\label{a1}
&&\frac{\rho_1}{2}\frac{d}{dt}\int^L_0\varphi_{ntt}^2dx+\kappa\int^L_0(\varphi_{nx}+\psi_{n})_t\varphi_{nttx}dx=0.
\end{eqnarray}
Differentiating \eqref{q1.1}$_2$, multiplying by
$\psi_{ntt}$, noting that
$$
\psi_{nttx}=\frac{1}{\kappa}\left(\rho_{1}\varphi_{ntttt}-\kappa\varphi_{mttxx}\right),
$$
 and then integrating the result over $(0, L)$, we get
\begin{eqnarray}\label{a2}
&&\frac{\rho_1\rho_2}{2\kappa}\frac{d}{dt}\int^L_0\varphi_{nttt}^2dx+\frac{\rho_2}{2}\frac{d}{dt}\int^L_0\varphi_{nttx}^2dx+\kappa\int^{L}_{0}
(\varphi_{nxt}+\psi_{nt})\psi_{ntt}dx\\
&&\quad+\frac{\alpha}{2}\frac{d}{dt}\int^{L}_{0}\psi_{nxt}^2dx+\xi_{1}\int^{L}_{0}
\theta_{nt}\psi_{nttx}dx+\xi_{2}\int^{L}_{0}
P_{nt}\psi_{nttx}dx =0.\nonumber
\end{eqnarray}
Differentiating $(\eqref{q1.1})_{3}, (\eqref{q1.1})_{4}$, multiplying by
$\theta_{nt}, P_{nt}$ respectively,
 and then integrating the result over $(0, L)$, we get
\begin{eqnarray}\label{a21}
&&\frac{c}{2}\frac{d}{dt}\int^L_0\theta_{nt}^2dx+\frac{r}{2}\frac{d}{dt}\int^{L}_{0}P_{nt}^2dx+d\frac{d}{dt}\int^{L}_{0}\theta_{nt}P_{nt}dx\nonumber\\
&&\quad+\xi_{1}\int^{L}_{0}
\theta_{nxt}\psi_{ntt}dx+\xi_{2}\int^{L}_{0}
P_{nxt}\psi_{ntt}dx\notag\\
&&+\kappa\int^{L}_{0}
\theta_{nxt}^{2}dx+h\int^{L}_{0}
P_{nxt}^{2}dx =0.
\end{eqnarray}
Combining
\eqref{a1}- \eqref{a2}, we get
\begin{eqnarray}\label{a3}
 \mathcal{G}_{n}(t)&+&\kappa\int_{0}^{t}\int_{0}^{L}\theta_{nxt}^{2}dx +h\int_{0}^{t}\int_{0}^{L}P_{nxt}^{2}dx \notag\\
    &=&\mathcal{G}_{n}(0),\nonumber
    \end{eqnarray}
    where
\begin{eqnarray}
    \mathcal{G}_{n}(t)&=&\frac{1}{2}\bigg[\rho_{1} \int_{0}^{L}\varphi_{ntt}^{2}dx +\frac{\rho_{1}\rho_{2}}{\kappa}\int_{0}^{L}\varphi_{nttt}^{2}dx+\rho_{2} \int_{0}^{L}\varphi_{nttx}^{2}dx\notag\\
    &&+\kappa\int_{0}^{L}(\varpi_{nxt}+\psi_{nt})^{2}dx
    + \alpha\int_{0}^{L}\psi_{nxt}^{2}dx\notag\\
    &&+c\int_{0}^{L}\theta_{tn}^{2}dx+r\int_{0}^{1}P_{tn}^{2}dx
    +2d\int_{0}^{L}\theta_{tn}P_{tn}dx\bigg].
    \end{eqnarray}
Similarly to the first a priori estimate, we can get there exists a
positive constant $C$ independent on $n$ such that
    \begin{eqnarray}
    \mathcal{G}_{n}(t)\leq C, \hspace{0.3cm}t\geq 0.\label{a31}
    \end{eqnarray}
\end{enumerate}
   {\bf Passing to the limit:}
    From (\ref{q1.15}) and \eqref{a31}, we conclude that for any $n \in \mathbb{N}$,
     \begin{eqnarray}
     \varphi_{n}&&is\hspace{0.1cm}bounded\hspace{0.1cm}
     in\hspace{0.2cm}L^{\infty}(\mathbb{R}_{+},
     H_{0}^{1})\notag\\
     \varphi_{nt}&&is\hspace{0.1cm}bounded\hspace{0.1cm} in\hspace{0.2cm}L^{\infty}(\mathbb{R}_{+}, L^{2} )
     \notag\\
     \varphi_{ntt}&&is\hspace{0.1cm}bounded\hspace{0.1cm} in\hspace{0.2cm}L^{\infty}(\mathbb{R}_{+}, L^{2} )\notag\\
     \psi_{n}&&is\hspace{0.1cm}bounded\hspace{0.1cm}
     in\hspace{0.2cm}L^{\infty}(\mathbb{R}_{+},
     H_{0}^{1})\notag\\
     \psi_{nt}&&is\hspace{0.1cm}bounded\hspace{0.1cm}
     in\hspace{0.2cm}L^{\infty}(\mathbb{R}_{+},
     L^{2})\notag\\
     \theta_{n}&&is\hspace{0.1cm}bounded\hspace{0.1cm}
     in\hspace{0.2cm}L^{\infty}(\mathbb{R}_{+},
     L^{2})\notag\\
     \theta_{nt}&&is\hspace{0.1cm}bounded\hspace{0.1cm}
     in\hspace{0.2cm}L^{\infty}(\mathbb{R}_{+},
     L^{2})\notag\\
     P_{n}&&is\hspace{0.1cm}bounded\hspace{0.1cm}
     in\hspace{0.2cm}L^{\infty}(\mathbb{R}_{+},
     L^{2})\notag\\
     P_{nt}&&is\hspace{0.1cm}bounded\hspace{0.1cm}
     in\hspace{0.2cm}L^{\infty}(\mathbb{R}_{+},
     L^{2}).
     \label{q1.16}
     \end{eqnarray}
Therefore, up to a subsequence, we observe that there exists a subsequence $(\varphi_\tau, \psi_\tau, \theta_\tau, P_\tau)$ of $(\varphi_n, \psi_n, \theta_n, P_n)$ and functions $(\varphi, \psi, \theta, P)$ that we may pass to the limit to obtain a weak solution with the above regularity by the by the fact that $L^{\infty}(\mathbb{R}_{+}, L^{2}) \to L^{2}(\mathbb{R}_{+}, L^{2})$ and $L^{\infty}(\mathbb{R}_{+}, H_{0}^{1}) \to L^{2}(\mathbb{R}_{+}, H_{0}^{1})$ as follow
      \begin{eqnarray}
     \varphi_{\tau}&& \rightharpoonup ^* \varphi \hspace{0.1cm}
     in\hspace{0.2cm}L^{2}(\mathbb{R}_{+},
     H_{0}^{1})\notag\\
     \varphi_{\tau t}&&\rightharpoonup ^* \varphi_{ t}\hspace{0.1cm} in\hspace{0.2cm}L^{2}(\mathbb{R}_{+}, L^{2} )
     \notag\\
     \varphi_{\tau tt}&&\rightharpoonup ^* \varphi_{  tt}\hspace{0.1cm} in\hspace{0.2cm}L^{2}(\mathbb{R}_{+}, L^{2} )\notag\\
     \psi_{\tau}&&\rightharpoonup ^* \psi\hspace{0.1cm}
     in\hspace{0.2cm}L^{2}(\mathbb{R}_{+},
     H_{0}^{1})\notag\\
     \psi_{\tau t}&&\rightharpoonup ^* \psi_t\hspace{0.1cm}
     in\hspace{0.2cm}L^{2}(\mathbb{R}_{+},
     L^{2})\notag\\
     \theta_{\tau}&&\rightharpoonup ^* \theta\hspace{0.1cm}
     in\hspace{0.2cm}L^{2}(\mathbb{R}_{+},
     L^{2})\notag\\
     \theta_{\tau t}&&\rightharpoonup ^* \theta_t\hspace{0.1cm}
     in\hspace{0.2cm}L^{2}(\mathbb{R}_{+},
     L^{2})\notag\\
     P_{\tau}&&\rightharpoonup ^*P\hspace{0.1cm}
     in\hspace{0.2cm}L^{2}(\mathbb{R}_{+},
     L^{2})\notag\\
     P_{\tau t}&&\rightharpoonup ^*P_t\hspace{0.1cm}
     in\hspace{0.2cm}L^{2}(\mathbb{R}_{+},
     L^{2}).\label{q1.17}
     \end{eqnarray}
     We then, by using the property of continuous of the operator in the distributions space and Lemma 1.4 in Kim \cite{Kim,Kim1}, can pass to limit the approximate problem (\ref{q1.1})-(\ref{q1.2}) and the desired results on problem (\ref{sys1.15})-(\ref{sys1.14}) is obtained.\\
    {\bf Continuous Dependence and Uniqueness:}
    Firstly we prove the continuous dependence and uniqueness for strong solutions of problem (\ref{sys1.15})-(\ref{sys1.14}).\\
    Let $(\varphi, \varphi_{t}, \varphi_{tt}, \psi, \psi_t, \theta, P)$ and $(\Gamma, \Gamma_{t}, \Gamma_{tt}, \Xi, \Xi_t, \Pi, \Omega)$ be two global solutions of (\ref{sys1.15})-(\ref{sys1.14}) with initial data $(\varphi_{0}, \varphi_{1}, \varphi_{2}, \psi_{0}, \psi_1, \theta_{0}, P_{0})$, $(\Gamma_{0}, \Gamma_{1}, \Gamma_{2}, \Xi_{0}, \Xi_{1}, \Pi_{0}, \Omega_{0})$ respectively.\\ Let
    \begin{eqnarray}
    \Lambda(t)=\varphi-\Gamma\notag\\
    \Sigma(t)=\psi-\Xi\notag\\
    \chi(t)=\theta-\Pi\notag\\
    M(t)=P-\Omega.
    \end{eqnarray}
    Then $(\Lambda, \Sigma, \chi, M)$ verifies (\ref{sys1.15})-(\ref{sys1.14}) and we have
    \begin{eqnarray}
    \left\{
    \begin{array}{ll}
     \rho_{1} \Lambda_{tt}-\kappa(\Lambda_{x}+\Sigma) _{x} \\
    -\rho_{2}\Lambda _{ttx}-\alpha \Sigma _{xx}+\kappa(\Lambda_{x}+\Sigma)
    -\xi_{1}\chi_{x}-\xi_{2} M_{x} =0 \\
    c\chi_{t}+d M_{t}-\kappa\chi_{xx}-\xi_{1}\Sigma_{tx}=0 \\
    d\chi_{t}+rM_{t}-hM_{xx}-\xi_{2}\Sigma_{tx}=0.
    \end{array}
    \right.\label{q1.19}
    \end{eqnarray}
    Multiplying $(\ref{q1.19})_{1}$ by $\Lambda_{t}$, $(\ref{q1.19})_{2}$ by $\Sigma_{t}$, $(\ref{q1.19})_{3}$ by $\chi_{t}$ and $(\ref{q1.19})_{4}$ by $M_t$. Integrating the results over $(0, L)$,
    and using the fact that
    \begin{eqnarray*}
    \kappa\int_{0}^{L}\Lambda_{tt}\Sigma_{tx}dx&=&\rho_{1}\int_{0}^{L}\Lambda_{ttt}\Lambda_{tt}dx+\kappa\int_{0}^{L}\Lambda_{txx}\Lambda_{tt}dx,
    \end{eqnarray*}
    we get
    \begin{eqnarray}
    &&\frac{d}{dt}\frac{1}{2}\bigg[\rho_{1} \int_{0}^{L}\Lambda_{t}^{2}dx +\frac{\rho_{1}\rho_{2}}{\kappa}\int_{0}^{L}\Lambda_{tt}^{2}dx+\rho_{2} \int_{0}^{L}\Lambda_{tx}^{2}dx+\kappa\int_{0}^{L}(\Lambda_{x}+\Sigma)^{2}dx\notag\\
    &&
    \quad+ \alpha\int_{0}^{L}\Sigma_{x}^{2}dx+c\int_{0}^{L}\chi^{2}dx+r\int_{0}^{L}M^{2}dx+2d\int_{0}^{1}\chi M dx\bigg]\notag\\
    &&+\kappa\int_{0}^{L}\chi_{x}^{2}dx +h\int_{0}^{L}M_{x}^{2}dx. \label{q1.20}
    \end{eqnarray}
     Then
    \begin{eqnarray*}
    \frac{d}{dt}\mathcal{E}(t)&\leq& 0,\\
    \end{eqnarray*}
    where
    \begin{eqnarray}
    \mathcal{E}(t)&=&\frac{1}{2}\bigg[\rho_{1} \int_{0}^{L}\Lambda_{t}^{2}dx +\frac{\rho_{1}\rho_{2}}{\beta}\int_{0}^{L}\Lambda_{tt}^{2}dx+\rho_{2} \int_{0}^{L}\Lambda_{tx}^{2}dx+\beta\int_{0}^{L}(\Lambda_{x}+\Sigma)^{2}dx\notag\\
    &&
    + \alpha\int_{0}^{L}\Sigma_{x}^{2}dx+c\int_{0}^{L}\chi^{2}dx+r\int_{0}^{L}M^{2}dx+2d\int_{0}^{1}\chi M dx\bigg]. \label{q1.21}
    \end{eqnarray}
   Integrating (\ref{q1.20}) to get
    \begin{eqnarray}
    \mathcal{E}(t)&\leq&\mathcal{E}(0)+ C_{1}\int^{t}_{0}\int_{0}^{L}(\vert\Lambda_{t}\vert^{2}
    +\vert\Lambda_{tt}\vert^{2}+\vert\Lambda_{tx}\vert^{2}
    +\vert\Sigma_{x}\vert^{2}+\vert(\Lambda_{x}+\Sigma)\vert^{2}\notag\\
    &&\hspace{0.2cm}+\vert\chi\vert^{2}+\vert M\vert^{2} )dx.\label{q1.22}
    \end{eqnarray}
    On the other hand, we have
    \begin{eqnarray}
    \mathcal{E}(t)&\geq&c_{0}\int_{0}^{L}(\vert\Lambda_{t}\vert^{2}
    +\vert\Lambda_{tt}\vert^{2}+\vert\Lambda_{tx}\vert^{2}
    +\vert\Sigma_{x}\vert^{2}+\vert(\Lambda_{x}+\Sigma)\vert^{2}\notag\\
    &&\hspace{0.2cm}+\vert\chi\vert^{2}+\vert M\vert^{2}  )dx.\label{q1.23}
    \end{eqnarray}
    Applying Gronwall's inequality to (\ref{q1.24}), we get
    \begin{eqnarray}
    &&\int_{0}^{L}(\vert\Lambda_{t}\vert^{2}
    +\vert\Lambda_{tt}\vert^{2}+\vert\Lambda_{tx}\vert^{2}
    +\vert\Sigma_{x}\vert^{2}+\vert(\Lambda_{x}+\Sigma)\vert^{2}\notag\\
    &&\hspace{0.2cm}
    +\vert\chi\vert^{2}+\vert M\vert^{2}  )dx\leq e^{C_{2}t}\mathcal{E}(0).\label{q1.24}
    \end{eqnarray}
    This shows that solution of problem (\ref{sys1.15})-(\ref{sys1.14}) depends continuously
    on the initial data.
    This ends the proof of Theorem \ref{the1}.
\end{proof}
    \section{Exponential stability}
    In this section, we will prove the exponential stability estimate for problem
    $(\ref{sys1.15})-(\ref{sys1.14})$ under assumption (\ref{eq9}), (\ref{con}) and by
    using a multiplier technique, with two different cases. In both cases, we find an exponential stability.\\
    The well posedness of the systems (\ref{sys1.151}) and (\ref{sys1.152}) can be obtained in a similar way from the previous section.
    \subsection{With frictional damping in the angular rotation}
    In the first problem we take frictional damping in the vertical displacement in the following system
    \begin{equation}
    \left\{
    \begin{array}{ll}
    \rho_{1} \varphi_{tt}-\kappa( \varphi_{x}+\psi) _{x}=0\\
    -\rho_{2}\varphi _{ttx}-\alpha \psi _{xx}+\kappa(\varphi_{x}+\psi)-\xi_{1}\theta_{x}-\xi_{2} P_{x}+\mu \psi_t =0\\
    c\theta_{t}+d P_{t}- \kappa \theta_{xx}-\xi_{1}\psi_{tx}=0\\
    d\theta_{t}+rP_{t}-hP_{xx}-\xi_{2}\psi_{tx}=0,
    \end{array}%
    \right. \label{sys1.151}
    \end{equation}
    where $\mu>0$.\\
    System (\ref{sys1.151}) is subjected with initial and Dirichlet boundary conditions (\ref{con1.1}),(\ref{sys1.14}).\\
    We state to use a several lemmas.
    \begin{lemma} \label{lem1}
    	The functional
    	\begin{eqnarray*}
    		F_{1}\left( t\right) &:&=-\rho _{1}\int_{0}^{L}\varphi _{t}\varphi dx, \\
    		&&
    	\end{eqnarray*}%
    	satisfies, for any $\varepsilon $ positive constant%
    	\begin{eqnarray}
    	F_{1}^{\prime }\left( t\right) &\leq &-\rho _{1}\int_{0}^{L}\varphi
    	_{t}^{2}dx+2\varepsilon C\int_{0}^{L}\varphi
    	_{x}^{2}dx+C\int_{0}^{L}(\varphi _{x}+\psi )^{2}dx,   \label{estF2.2.10.0.0}
    	\end{eqnarray}%
    	where $C$ positive constant.
    \end{lemma}

    \begin{proof}
    	Differentiating $F_{1}$, using integrating by parts and (\ref{sys1.14}), we
    	get
    	\begin{eqnarray}
    	F_{1}^{\prime }\left( t\right) &=&-\rho _{1}\int_{0}^{L}\varphi
    	_{t}^{2}dx-\rho _{1}\int_{0}^{L}\varphi _{tt}\varphi dx  \notag \\
    	&=&-\rho _{1}\int_{0}^{L}\varphi _{t}^{2}dx+\kappa \int_{0}^{L}(\varphi
    	_{x}+\psi )\varphi _{x}dx. \nonumber
    	\end{eqnarray}%
    	Using Young's inequality, we obtain (\ref{estF2.2.10.0.0})
    \end{proof}

    \begin{lemma}\label{lem2}
    	The functional
    	\begin{equation*}
    		F_{2}\left( t\right) :=\rho _{1}\int_{0}^{L}\varphi _{t}\varphi dx+\frac{\mu
    		}{2}\int_{0}^{L}\psi ^{2}dx,
    	\end{equation*}%
    	satisfies,
    	\begin{eqnarray}
    	F_{2}^{\prime }(t) &\leq &\rho _{1}\int_{0}^{L}\varphi _{t}^{2}dx-\frac{\alpha}{4}%
    	\int_{0}^{L}\psi _{x}^{2}dx-\kappa \int_{0}^{L}(\varphi _{x}+\psi )^{2}dx
    	\notag \\
    	&&+\frac{\zeta _{1}^{2}}{\alpha}\int_{0}^{L}\theta _{x}^{2}dx+\frac{\zeta _{2}^{2}%
    	}{\alpha}\int_{0}^{L}P_{x}^{2}dx +\frac{\rho_{2}^{2}}{\alpha}\int_{0}^{L}\varphi
    _{tt} ^{2}dx.  \label{Estimation F_2}
    	\end{eqnarray}
    \end{lemma}

    \begin{proof}
    	Differentiating $F_{2}$, using integrating by parts and (\ref{sys1.14}), we
    	get
    	\begin{eqnarray}
    	F_{2}^{\prime }\left( t\right) &=& \rho _{1}\int_{0}^{L}\varphi
    	_{t}^{2}dx+\rho _{1}\int_{0}^{L}\varphi
    	_{tt}\varphi dx+\mu \int_{0}^{L}\psi \psi _tdx  \nonumber\\
    	&&=\rho _{1}\int_{0}^{L}\varphi
    	_{t}^{2}dx-\kappa \int_{0}^{L}(\varphi _{x}+\psi )\varphi _{x}dx -\rho_{2} \int_{0}^{L}\varphi
    	_{tt} \psi_{x}dx  \notag \\
    	&&-\alpha\int_{0}^{L}\psi _{x}^{2}dx-\kappa \int_{0}^{L}(\varphi _{x}+\psi )\psi
    	dx  \notag \\
    	&&-\zeta _{1}\int_{0}^{L}\theta \psi _{x}dx-\zeta _{2}\int_{0}^{L}P\psi
    	_{x}dx.  \label{Derivee F_2}
    	\end{eqnarray}%
    	Thanks to Young's inequality, we have
     \begin{equation}
    	-\rho_{2} \int_{0}^{L}\varphi
    	_{tt} \psi_{x}dx\leq \frac{\alpha}{4}\int_{0}^{L}\psi
    	_{x}^{2}dx+\frac{\rho_{2}^{2}}{\alpha}\int_{0}^{L}\varphi
    	_{tt} ^{2}dx,
    	\label{suite derivee F_21}
    	\end{equation}
    	and
    	\begin{equation}
    	\zeta _{1}\int_{0}^{L}\theta \psi_x dx\leq \frac{\alpha}{4}\int_{0}^{L}\psi
    	_{x}^{2}dx+\frac{\zeta _{1}^{2}}{\alpha}\int_{0}^{L}\theta ^{2}dx,
    	\label{suite derivee F_2}
    	\end{equation}%
    	and%
    	\begin{equation}
    		\zeta _{2}\int_{0}^{L}P\psi _{x}dx\leq \frac{\alpha}{4}\int_{0}^{L}\psi
    		_{x}^{2}dx+\frac{\zeta _{2}^{2}}{\alpha}\int_{0}^{L}P^{2}dx.\label{suite derivee F_3}
    	\end{equation}%
    	By (\ref{suite derivee F_21}), (\ref{suite derivee F_2}), and (\ref%
    	{suite derivee F_21}) together with (\ref{Derivee F_2}), we arrive to the proof of Lemma \ref{lem2}.
    \end{proof}

    \begin{lemma} \label{lem3}
    	The functional
    	\begin{equation*}
    		F_{3}\left( t\right) :=-\rho _{2}\int_{0}^{L}\varphi _{tx}\varphi _{x}dx+%
    		\frac{\mu }{2}\int_{0}^{L}\psi ^{2}dx,
    	\end{equation*}%
    	satisfies,
    	\begin{eqnarray}
    	F_{3}^{\prime }(t) &\leq &-\rho _{2}\int_{0}^{L}\varphi _{tx}^{2}dx-\frac{%
    		\rho _{1}\rho _{2}}{\kappa }\int_{0}^{L}\varphi _{tt}^{2}dx- \alpha\Big(1-\frac{3c_*^2}{4}\Big)
    	\int_{0}^{L}\psi _{x}^{2}dx  \notag \\
    	&&+\frac{\zeta _{1}^{2}}{\alpha}\int_{0}^{L}\theta
    	_{x}^{2}dx+\frac{\zeta _{2}^{2}}{\alpha}\int_{0}^{L}P_{x}^{2}dx +\frac{\kappa^{2}}{\alpha}\int_{0}^{L}(\varphi_x+\psi) ^{2}dx,
    	\label{Estimation F_3}
    	\end{eqnarray}%
    	where $c$ positive constant
    \end{lemma}

    \begin{proof}
    	Differentiating $F_{3}$, using integrating by parts and (\ref{sys1.14}), we get
    	\begin{eqnarray}
    	F_{3}^{\prime }\left( t\right) &=&-\rho _{2}\int_{0}^{L}\varphi _{ttx}\varphi
    	_{x}dx-\rho _{2}\int_{0}^{L}\varphi _{tx}^{2}dx +\mu \int_{0}^{L} \psi_t\psi dx  \notag \\
    	&=& \rho _{2}\int_{0}^{L}\varphi _{tt}\varphi
    	_{xx}dx-\rho _{2}\int_{0}^{L}\varphi _{tx}^{2}dx \notag \\
   	&+&   \int_{0}^{L}(\rho_2 \varphi_{ttx}+b\psi_{xx}-\kappa(\varphi_x+\psi)+\xi_1\theta_x+\xi_2P_x)\psi dx.    	\label{Djamel}
    	\end{eqnarray}%
    	 From (\ref{sys1.15})$_1$, we have $\psi _{x}=\frac{\rho _{1}}{\kappa }\varphi _{tt}-\varphi _{xx},$
    	then
    	\begin{eqnarray}
    -\rho _{2}\int_0^L\varphi_{tt}\varphi_{xx}dx+\rho_2\int_0^L\varphi_{tt}\psi_{xt}dx=\frac{\rho_1\rho_2}{\kappa }\int_0^L\varphi_{tt}^2dx. \nonumber
    	\end{eqnarray}
    	Then
    	\begin{eqnarray}
    	F_{3}^{\prime }\left( t\right)
    	&=&  -\rho _{2}\int_{0}^{L}\varphi _{tx}^{2}dx+\frac{\rho _{1}\rho _{2}}{%
    		\kappa }\int_{0}^{L}\varphi _{tt}^{2}dx -  \alpha \int_{0}^{L} \psi^2_{x}dx-2\rho_{2}\int_{0}^{L}\varphi _{tt} \psi _xdx\notag \\
    	&-&\kappa\int_{0}^{L}(\varphi_x+\psi) \psi dx+\int_{0}^{L}( \xi_1\theta_x+\xi_2P_x)\psi dx.
    	\label{Djamel1}
    	\end{eqnarray}%
    	  	Thanks to Young and Poincr\'{e}'s inequalities, we have
    	 \begin{equation}
    	 	-\kappa \int_{0}^{L}(\varphi_x+\psi) \psi dx\leq \frac{\alpha c_*^2}{4}\int_{0}^{L}\psi
    	 	_{x}^{2}dx+\frac{\kappa^{2}}{\alpha}\int_{0}^{L}(\varphi_x+\psi) ^{2}dx,
    	 	\label{s22}
    	 	\end{equation}
    	 and
    	 \begin{equation}
    	 \zeta _{1}\int_{0}^{L}\theta_x \psi dx\leq \frac{\alpha c_*^2}{4}\int_{0}^{L}\psi
    	 _{x}^{2}dx+\frac{\zeta _{1}^{2}}{\alpha}\int_{0}^{L}\theta_x ^{2}dx,
    	 \label{s23}
    	 \end{equation}%
    	 and%
    	 \begin{equation}
    	 \zeta _{2}\int_{0}^{L}P _{x}\psi dx\leq \frac{\alpha c_*^2}{4}\int_{0}^{L}\psi
    	 _{x}^{2}dx+\frac{\zeta _{2}^{2}}{\alpha}\int_{0}^{L}P_x^{2}dx.\label{s24}
    	 \end{equation}
    	 By (\ref{s22}), (\ref{s23}) and (\ref%
    	 {s24}) together with (\ref{Djamel1}), we arrive to the proof of Lemma \ref{lem3}.
    \end{proof}

    \begin{lemma} \label{lem4}
    	The functional
    	\begin{equation}
    	F_{4}\left( t\right) :=-\rho _{2}\int_{0}^{L}\varphi _{xt}\left( \varphi
    	_{x}+\psi \right) dx-\frac{\alpha\rho _{1}}{\kappa }\int_{0}^{L}\varphi _{xt}\psi
    	dx,
    	\end{equation}%
    	satisfies
    	\begin{eqnarray}
    	F_{4}^{\prime }\left( t\right) &\leq& -\frac{\rho _{2}}{2}\int_{0}^{L}\varphi
    	_{xt}^{2}dx-\frac{\kappa }{2}\int_{0}^{L}\left( \varphi _{x}+\psi \right)
    	^{2}dx+c\int_{0}^{L}\psi _{t}^{2}dx\nonumber\\
    	&+&\frac{\zeta _{1}^{2}}{4}%
    	\int_{0}^{L}\theta _{x}^{2}dx+\frac{\zeta _{2}^{2}}{4}\int_{0}^{L}P_{x}^{2}dx,
    	\label{Estimation number 4}
    	\end{eqnarray}%
    	where $c$ a positive constant.
    \end{lemma}

    \begin{proof}
    	Direct computation using integration by parts, we get
    	\begin{eqnarray}
    	F_{4}^{\prime }\left( t\right) &=&-\rho _{2}\int_{0}^{L}\varphi _{ttx}\left(
    	\varphi _{x}+\psi \right) dx-\rho _{2}\int_{0}^{L}\varphi _{xt}\left(
    	\varphi _{x}+\psi \right) _{t}dx-\frac{\alpha\rho _{1}}{\kappa }%
    	\int_{0}^{L}\varphi _{xt}\psi_t  dx  \notag\\
    	&&  -\frac{b\rho _{1}}{\kappa }%
    	\int_{0}^{L}\varphi _{xtt}\psi dx.  \label{Estimation F_4}
    	\end{eqnarray}%
    	Multiplying (\ref{sys1.11})$_2$ by $\left( \varphi _{x}+\psi \right),$ we get
    	\begin{eqnarray}
    	&&-\rho _{2}\int_{0}^{L}\varphi _{ttx}\left( \varphi _{x}+\psi \right)
    	dx+\alpha\int_{0}^{L}\psi _{x}\left( \varphi _{x}+\psi \right) _{x}dx +\kappa
    	\int_{0}^{L}\left( \varphi _{x}+\psi \right)^2dx \nonumber\\
    	&&=\zeta _{1}\int_{0}^{L}\theta _{x}\left( \varphi _{x}+\psi \right)
    	dx+\zeta _{2}\int_{0}^{L}P_{x}\left( \varphi _{x}+\psi \right) dx\nonumber\\
    	&&-\mu\int_{0}^{L}\psi_t\left( \varphi _{x}+\psi \right) dx.
    	\label{Khaled Zennir}
    	\end{eqnarray}
    	By Young's inequality, we get
    	\begin{equation}
    -\mu\int_{0}^{L}\psi_t\left( \varphi _{x}+\psi \right) dx\leq
    	\frac{\kappa}{4}\int_{0}^{L}\left( \varphi _{x}+\psi \right)^{2}dx+\frac{\mu^2}{4}%
    	\int_{0}^{L}\psi_t ^{2}dx,  \label{Choucha0}
    	\end{equation}
    	and
    	\begin{equation}
    	\zeta _{1}\int_{0}^{L}\theta _{x}\left( \varphi _{x}+\psi \right) dx\leq
    	\frac{\zeta _{1}^{2}}{4}\int_{0}^{L}\theta _{x}^{2}dx+\frac{\kappa }{4}%
    	\int_{0}^{L}\left( \varphi _{x}+\psi \right) ^{2}dx,  \label{Choucha1}
    	\end{equation}
    	and%
    	\begin{equation}
    	-\zeta _{2}\int_{0}^{L}P_{x}\left( \varphi _{x}+\psi \right) dx\leq \frac{%
    		\zeta _{2}^{2}}{4}\int_{0}^{L}P_{x}^{2}dx+\frac{\kappa }{4}%
    	\int_{0}^{L}\left( \varphi _{x}+\psi \right) ^{2}dx.  \label{Choucha2}
    	\end{equation}%
    	 Then, by (\ref{Choucha0}), (\ref{Choucha1}) and (\ref{Choucha2}) with (\ref{Khaled Zennir}) we have
    	\begin{eqnarray}
    	&&-\rho _{2}\int_{0}^{L}\varphi _{ttx}\left( \varphi _{x}+\psi \right)
    	dx+\alpha\int_{0}^{L}\psi _{x}\left( \varphi _{x}+\psi \right) _{x}dx
    	\label{Khaled Zennir1}\\
    	&&\leq\frac{-3\kappa }{4}%
    	\int_{0}^{L}\left( \varphi _{x}+\psi \right) ^{2}dx	+\frac{\zeta _{1}^{2}}{4}\int_{0}^{L}\theta _{x}^{2}dx+\frac{%
    		\zeta _{2}^{2}}{4}\int_{0}^{L}P_{x}^{2}dx+\frac{\mu^2}{4}%
    	\int_{0}^{L}\psi_t ^{2}dx.\nonumber
    	\end{eqnarray}
    From (\ref{sys1.15})$_1$, we have $\left( \varphi _{x}+\psi \right) _{x}=\dfrac{\rho _{1}}{\kappa }\varphi _{tt}$,
    then
     	\begin{eqnarray}
     &&-\rho _{2}\int_{0}^{L}\varphi _{ttx}\left( \varphi _{x}+\psi \right)
     dx+\dfrac{\alpha\rho _{1}}{\kappa }\int_{0}^{L}\psi _{x}\varphi _{tt}dx
\label{A1}\\
     &&\leq\frac{-3\kappa }{4}%
     \int_{0}^{L}\left( \varphi _{x}+\psi \right) ^{2}dx	+\frac{\zeta _{1}^{2}}{4}\int_{0}^{L}\theta _{x}^{2}dx+\frac{%
     	\zeta _{2}^{2}}{4}\int_{0}^{L}P_{x}^{2}dx+\frac{\mu^2}{4}%
     \int_{0}^{L}\psi_t ^{2}dx.     \nonumber
     \end{eqnarray}
    	Using Young's inequality, we have
    	\begin{eqnarray}
    	&&	-\rho _{2}\int_{0}^{L}\varphi _{xt}\left( \varphi _{x}+\psi \right) _{t}dx-%
    		\frac{\alpha\rho _{1}}{\kappa }\int_{0}^{L}\varphi _{xt}\psi _{t}dx\nonumber\\ &&=-\rho
    		_{2}\int_{0}^{L}\varphi _{xt}^{2}dx-\rho _{2}\int_{0}^{L}\varphi _{xt}\psi
    		_{t}dx-\frac{\alpha\rho _{1}}{\kappa }\int_{0}^{L}\varphi _{xt}\psi _{t}dx \nonumber\\
    	&	&\leq -\rho _{2}\int_{0}^{L}\varphi _{xt}^{2}dx+\left( \rho _{2}+\frac{%
    			\alpha^{2}\rho _{1}^{2}}{\rho _{2}\kappa }\right) \int_{0}^{L}\psi _{t}^{2}dx.\label{A2}
    	\end{eqnarray}%
    	Therefore, by (\ref{Khaled Zennir1}), (\ref{A1}) and (\ref{A2}), the desired result is obtained.
    \end{proof}

    \begin{lemma}
    	There exists a constant $\beta _{0}>0$ such that%
    	\begin{equation*}
    		\left( N-\beta _{0}\right) \mathcal{E}\left( t\right) \leq \mathcal{L}\left(
    		t\right) \leq \left( N+\beta _{0}\right) \mathcal{E}\left( t\right) \text{ \
    			\ }\forall t\geq 0,
    	\end{equation*}%
    	where $\mathcal{L}\left( t\right) $ is a Lyapunov functional defined by
    \end{lemma}

    \begin{equation}
    \mathcal{L}\left( t\right) :=N\mathcal{E}\left( t\right)
    +N_{1}F_{1}(t)+F_{2}(t)+F_{3}(t)+F_{4}(t),
    \label{Functional L(t)}
    \end{equation}%
    and $N, N_{1}>\beta _{0}$ is a sufficiently large
    constant.
    \begin{proof}
    	It follows from Young, Poincar\'{e} and Cauchy-Schwarz's inequalities that%
    	\begin{eqnarray}
    		\left\vert F_{1}\left( t\right) \right\vert  &\leq &\frac{\rho _{1}}{2}%
    		\int_{0}^{L}\varphi _{t}^{2}dx+\frac{\rho _{1}}{2}\int_{0}^{L}\varphi ^{2}dx \nonumber\\
    		\left\vert F_{2}\left( t\right) \right\vert  &\leq &\frac{\rho _{1}}{2}%
    		\int_{0}^{L}\varphi _{t}^{2}dx+\frac{\rho _{1}}{2}\int_{0}^{L}\varphi ^{2}dx+\frac{\mu }{2}%
    		\int_{0}^{L}\psi ^{2}dx \nonumber\\
    		\left\vert F_{3}\left( t\right) \right\vert  &\leq &\rho
    		_{2}\int_{0}^{L}\varphi _{xt}^{2}dx+\frac{\rho _{2}}{2}\int_{0}^{L}\varphi
    		_{x}^{2}dx+\frac{\mu }{2}\int_{0}^{L}\psi ^{2}dx \nonumber\\
    		\left\vert F_{4}\left( t\right) \right\vert  &\leq &\frac{\rho _{2}}{2}%
    		\int_{0}^{L}\varphi _{xt}^{2}dx+\frac{\rho _{2}}{2}\int_{0}^{L}\left(
    		\varphi _{x}+\psi \right) ^{2}dx+\frac{\alpha\rho _{1}}{2\kappa }
    		\int_{0}^{L}\varphi _{xt}^{2}dx+\frac{\alpha\rho _{1}}{2\kappa }\int_{0}^{L}\psi
    		^{2}dx.\nonumber
    	\end{eqnarray}
    	Thus, there exists a constant $$\beta _{0}=\max\{N_1\rho_{1}, \rho_{1}+\frac{\rho_{2}}{2}+\frac{\mu}{2}, \frac{3\rho_{2}}{2}+\frac{\mu}{2},  \rho_{2}+\frac{\alpha\rho_{1}}{\kappa}\}>0,$$ such that
    	\begin{equation*}
    		\left\vert \mathcal{L}\left( t\right) -NE\left( t\right) \right\vert
    		=\left\vert N_{1}F_{1}(t)+F_{2}(t)+ F_{3}(t)+ F_{4}(t)\right\vert
    		\leq \beta _{0}\mathcal{E}\left( t\right).
    	\end{equation*}
    \end{proof}

    \begin{theorem}\label{the2}
    	Assume that (\ref{eq9}), (\ref{con}) hold. Then, there exist a positive constants $\lambda _{1}$ and $\lambda _{2}$
    	such that the energy functional (\ref{energy}) satisfies
    	\begin{equation}
    	\mathcal{E}(t)\leq \lambda _{2}e^{-\lambda _{1}t},\forall t\geq 0.
    	\label{eq2.26.0}
    	\end{equation}
    \end{theorem}
    	
    	\begin{proof}
    		It follows from Lemma \ref{lem1}- Lemma \ref{lem4} and using (\ref%
    		{Functional L(t)}), that for any $t>0$ and by using the inequality%
    		\begin{equation*}
    			\left( \varphi _{x}+\psi \right) ^{2}\leq 2\varphi _{x}^{2}+2\psi ^{2},
    		\end{equation*}
    	 we have
    		\begin{eqnarray}
    		\mathcal{L}^{\prime }\left( t\right)  &\leq &\left( -\delta N+\frac{2\zeta
    			_{1}^{2}}{\alpha}+c+\frac{\zeta _{1}^{2}}{4}\right) \int_{0}^{L}\theta
    		_{x}^{2}dx+\left( -hN+\frac{2\zeta _{2}^{2}}{\alpha}+c+\frac{\zeta _{2}^{2}}{4}%
    		\right) \int_{0}^{L}P_{x}^{2}dx  \nonumber \\
    		&&+\left( -\rho _{1}N_{1}-\rho _{1}\right) \int_{0}^{L}\varphi
    		_{t}^{2}dx+\left( -\mu N+c\right) \int_{0}^{L}\psi _{t}^{2}dx\label{derivee L}\\
    		&&+ \left( -\frac{%
    			3\kappa }{2}+N_{1}c+\frac{\kappa^2}{2}+\frac{\kappa^2}{\alpha}\right) \int_{0}^{L}\left( \varphi _{x}+\psi \right)
    		^{2}dx  \notag \\
    		&&-\alpha(\frac{1}{2}+(1-\frac{3c^2_*}{4}))\int_{0}^{L}\psi _{x}^{2}dx-\frac{3\rho _{2}}{2}%
    		\int_{0}^{L}\varphi _{tx}^{2}dx+(\frac{\rho_{2}^2}{\alpha}-\frac{\rho_{2}\rho_{1}}{\kappa}) \int_{0}^{L}\psi _{tt}^{2}dx. \notag
    		\end{eqnarray}%
    		We choose $N_{1}>1$ such that
    		\begin{equation*}
    			0>-\frac{%
    				3\kappa }{2}+N_{1}c+\frac{\kappa^2}{2}+\frac{\kappa^2}{\alpha},
    		\end{equation*}
    	and $c_*, \alpha, \kappa, \rho_{1}, \rho_{2}$ such that
    	$$
    	1-\frac{3c^2_*}{4}>0, \frac{\rho_{2}^2}{\alpha}-\frac{\rho_{2}\rho_{1}}{\kappa}<0.
    	$$
    		Then, we choose $N$ large enough such that
    		\begin{eqnarray*}
    			0 &>&-\delta N+\frac{2\zeta
    				_{1}^{2}}{\alpha}+c+\frac{\zeta _{1}^{2}}{4}\\
    			0 &>&-\mu N+c.
    		\end{eqnarray*}%
    		Thus, there exists a positive constant $\lambda >0$ such that
    		\begin{equation}
    		\mathcal{L}^{\prime }\left( t\right) \leq -\lambda \mathcal{E}(t).
    		\label{El wardi}
    		\end{equation}%
    		Integrating (\ref{El wardi}) over $\left( 0, t\right) $ and using (\ref%
    		{derivee L}), we can obtain that there exist two positive constants $\lambda
    		_{1}$ and $\lambda _{2}$ such that%
    		\begin{equation*}
    			\mathcal{E}(t)\leq \lambda _{2}e^{-\lambda _{1}t},\forall t\geq 0,
    		\end{equation*}%
    		which completes the proof.
    	\end{proof}
    \subsection{ With frictional damping in the vertical displacement}
    In this subsection, we consider
    \begin{equation}
    \left\{
    \begin{array}{ll}
    \rho _{1}\varphi _{tt}-\kappa (\varphi _{x}+\psi )_{x}+\mu\varphi_{t}=0 \\
    -\rho _{2}\varphi _{ttx}-b\psi _{xx}+\kappa (\varphi _{x}+\psi )-\xi
    _{1}\theta _{x}-\xi _{2}P_{x}=0 \\
    \tau _{0}\theta _{t}+dP_{t}-\delta \theta _{xx}-\xi _{1}\psi _{tx}=0 \\
    d\theta _{t}+rP_{t}-hP_{xx}-\xi _{2}\psi _{tx}=0,
    \end{array}%
    \right.  \label{sys1.152}
    \end{equation}
    where $\mu>0$, with initial conditions
    \begin{equation}
    \left\{
    \begin{array}{l}
    \varphi \left( x, 0\right) =\varphi _{0}\left( x\right), \varphi _{t}\left(
    x, 0\right) =\varphi _{1}\left( x\right), \varphi _{tt}\left( x, 0\right)
    =\varphi _{2}\left( x\right) \\
    \varphi _{ttt}(x, 0)=\varphi _{3}(x),\ \psi \left( x, 0\right) =\psi
    _{0}\left( x\right), \\
    \theta (x, 0)=\varphi _{0}(x),\ P\left( x, 0\right) =P_{0}\left( x\right), x\in \left( 0, 1\right),
    \end{array}%
    \right.  \label{con1.12}
    \end{equation}%
    where $\varphi _{0}, \varphi _{1}, \varphi _{2}, \varphi _{3}, \psi _{0}, \theta
    _{0}, P_{0}$,  are given functions (satistify the assumption of the previous section), and the Dirichlet and Newmann conditions (\ref{sys1.14}).

    \begin{lemma}
    	Define the energy of solution as
    	\begin{eqnarray}
    	\mathcal{E}\left( t\right) &=&\frac{1}{2}\int_{0}^{L}\left[ \rho _{1}\varphi
    	_{t}^{2}+\alpha \psi _{x}^{2}+\kappa (\varphi _{x}+\psi )^{2}+\frac{\rho
    		_{1}\rho _{2}}{\kappa }\varphi _{tt}^{2}+\rho _{2}\varphi _{tx}^{2}\right] dx
    	\notag \\
    	&&+\frac{1}{2}\int_{0}^{L}\left[ \tau _{0}\theta ^{2}+rP^{2}+2d\theta P%
    	\right] dx,  \label{eq2.2.0}
    	\end{eqnarray}%
    	satisfies
    	\begin{equation}
    	\mathcal{E}^{\prime }\left( t\right) =-\delta \int_{0}^{L}\theta
    	_{x}^{2}dx-h\int_{0}^{L}P_{x}^{2}dx-\mu\int_{0}^{L}\varphi _{t}^{2}dx-\frac{\mu\rho_{2}}{\kappa}\int_{0}^{L}\varphi _{tt}^{2}dx\leq 0. \label{eq2.2}
    	\end{equation}
    \end{lemma}

    \begin{proof}
    	Multiplying equations of (\ref{sys1.152}) by $\varphi _{t}, \psi_{t}, \theta, P$ respectively, using integration by parts, and (\ref{sys1.14}), we get
    	\begin{equation}
    	\left\{
    	\begin{array}{l}
    	\dfrac{\rho _{1}}{2}\dfrac{d}{dt}\displaystyle\int_{0}^{L}\varphi
    	_{t}^{2}dx+\kappa \displaystyle\int_{0}^{L}(\varphi _{x}+\psi )\varphi
    	_{tx}dx+\mu
    	\displaystyle\int_{0}^{L}\varphi _{t}^{2}dx=0 \\
    	+\rho _{2}\displaystyle\int_{0}^{L}\varphi _{tt}\psi _{tx}dx+\frac{\alpha }{2%
    	}\frac{d}{dt}\displaystyle\int_{0}^{L}\psi _{x}^{2}dx+\kappa \displaystyle%
    	\int_{0}^{L}(\varphi _{x}+\psi )\psi _{t}dx \\
    	\hspace{1cm}-\xi _{1}\displaystyle\int_{0}^{L}\theta _{x}\psi _{t}dx-\xi _{2}%
    	\displaystyle\int_{0}^{L}P_{x}\psi _{t}dx =0 \\
    	\dfrac{\tau _{0}}{2}\dfrac{d}{dt}\displaystyle\int_{0}^{L}\theta ^{2}dx+d%
    	\displaystyle\int_{0}^{L}P_{t}\theta dx+\delta \displaystyle%
    	\int_{0}^{L}\theta _{x}^{2}dx-\xi _{1}\displaystyle\int_{0}^{L}\psi
    	_{xt}\theta dx=0 \\
    	\dfrac{r}{2}\dfrac{d}{dt}\displaystyle\int_{0}^{L}P^{2}dx+d\displaystyle%
    	\int_{0}^{L}\theta _{t}Pdx+h\displaystyle\int_{0}^{L}P_{x}^{2}dx-\xi _{2}%
    	\displaystyle\int_{0}^{L}\psi _{xt}Pdx=0.
    	\end{array}%
    	\right.  \label{sys1.16.02}
    	\end{equation}%
    	Now, taking the derivative $(\ref{sys1.152})_{1}$, we get%
    	\begin{equation}
    	\psi _{tx}=\rho _{1}\frac{\left( \varphi _{tt}\right) _{t}}{\kappa }-
    	\left( \varphi _{x}\right) _{tx}+\frac{\mu}{\kappa}\varphi_{tt}. \label{Changement}
    	\end{equation}%
    	Now substituting $(\ref{Changement})$ in $(\ref{sys1.152})_{2}$ using
    	integration by parts and summing. Then we obtain $\mathcal{E}$ is decreasing.
    \end{proof}

\begin{lemma}
	The functional%
	\begin{equation}
	F_{1}\left( t\right) :=-\frac{\mu}{2}\int_{0}^{1}\varphi_{t}^{2} dx-\kappa \int_{0}^{1}\varphi_{tx}\varphi_{x} dx,
	\end{equation}
	satisfies%
	\begin{eqnarray}
	F_{1}^{\prime }\left( t\right) &\leq &-\kappa\int_{0}^{1}\varphi
	_{tx}^{2}dx+\varepsilon_{1}\int_{0}^{1}\psi
	_{x}^{2}dx+c(1+\frac{1}{\varepsilon_{1}})\int_{0}^{1}\varphi_{tt}
	^{2}dx.  \label{estF1.2.5}
	\end{eqnarray}
\end{lemma}
\begin{proof}
	Direct computation using integration by parts, we get
	\begin{eqnarray}
	F_{1}^{\prime }\left( t\right) &=& \rho_{1} \int_{0}^{1}\varphi_{tt}^{2}dx-\kappa \int_{0}^{1}\psi_{x}\varphi_{tt}dx-\kappa\int_{0}^{1} \varphi_{tx}^{2} dx, \notag\label{eq2.6}
	\label{bib2}
	\end{eqnarray}
	estimate (\ref{estF1.2.5}) easily follows by using Young's and Poincare's inequalities.
\end{proof}
\begin{lemma}
	The functional%
	\begin{eqnarray*}
		F_{2}\left( t\right) &:=&\rho_{1}\int_{0}^{1}\varphi\varphi_{t}dx+\frac{\mu}{2}\int_{0}^{1}\varphi^{2}dx+
		\frac{\mu\rho_{2}}{2\kappa}\int_{0}^{1}\varphi_{t}^{2}dx+\rho_{2}\int_{0}^{1}\varphi_{tx}\varphi_{x}dx,
	\end{eqnarray*}
	satisfies,
	\begin{eqnarray}
	F_{2}(t) &\leq&-\frac{\rho_{1}\rho_{2}}{\kappa}\int_{0}^{1}\varphi_{tt}^{2}dx-\frac{\kappa}{2}\int_{0}^{1}(\varphi_{x}+\psi)^{2}dx-\frac{\alpha}{2}\int_{0}^{1}\psi_{x}^{2}dx\notag\\
	&& +\rho_{2}\int_{0}^{1}\varphi_{tx}^{2}dx +\rho_{1}\int_{0}^{1}\varphi_{t}^{2}dx+c\int_{0}^{1}\theta_{x}^{2}dx+c\int_{0}^{1}P_{x}^{2}dx.\label{estF2.2.10}
	\end{eqnarray}%
\end{lemma}
\begin{proof}
	Differentiating $F_{2}$, using
	integrating by parts and (\ref{sys1.14}), we get%
	\begin{eqnarray}
	F_{2}^{\prime }\left( t\right) &=&\rho_{1}\int_{0}^{1}\varphi_{t}^{2}dx-\kappa\int_{0}^{1}(\varpi_{x}+\psi)^{2}dx-\alpha\int_{0}^{1}\psi_{x}^{2}dx  \notag  \label{eq2.11}\\
	&&-\frac{\rho_{1}\rho_{2}}{\beta}\int_{0}^{1}\varphi_{tt}^{2}dx +\xi_{1}\int_{0}^{1}\theta_{x}\psi dx\notag\\
	&&+\rho_{2}\int_{0}^{1}\varpi_{tx}^{2}dx+\xi_{2}\int_{0}^{1}P_{x}\psi dx.
	\end{eqnarray}
	Using Young and Poincar\'{e}'s inequalities, we obtain (\ref{estF2.2.10})
\end{proof}

\begin{theorem} \label{the3}
	Assume (\ref{eq9}), there exist positive constants $\lambda_{1}$ and $\lambda_{2}$ such that the energy functional
	(\ref{eq2.2.0}) satisfies%
	\begin{equation}
	\mathcal{E}(t) \leq \lambda_{2}e^{-\lambda_{1}t}, \forall t\geq 0.  \label{eq2.26}
	\end{equation}
\end{theorem}
	\begin{proof}
		We define a Lyapunov functional%
		\begin{equation}
		\mathcal{L}\left( t\right) :=N\mathcal{E}\left( t\right)+N_{1}F_{1}(t)+F_{2}(t),  \label{eq2.27}
		\end{equation}
		where $N, N_{1}>0$.
		By differentiating (\ref{eq2.27}) and using (\ref{eq2.2}), (\ref{estF1.2.5}%
		) and (\ref{estF2.2.10})  we have
		\begin{eqnarray*}
			\mathcal{L}^{\prime }\left( t\right) &\leq & -\left[ N\mu-\rho_{2}\right] \int_{0}^{1}\varphi_{t}^{2}dx -\left[ \frac{\mu\rho_{2}}{\kappa}N+\frac{\rho_{1}\rho_{2}}{\beta}-cN_{1}(1+\frac{1}{\varepsilon_{1}})\right] \int_{0}^{1}\varphi _{tt}^{2}dx\\
			&&-\left[ \frac{\alpha}{2}-\varepsilon_{1}N_{1}\right]\int_{0}^{1}\psi_{x}^{2} dx-\kappa
			\int_{0}^{1}(\varphi_{x}+\psi)^{2}dx\\
			&&-\left[ \kappa N_{1}-\rho_{2} \right] \int_{0}^{1}\varphi_{tx}^{2}dx-\left[ Nh-c\right] \int_{0}^{1}P_{x}^{2}dx-\left[ N\delta-c\right] \int_{0}^{1}\theta_{x}^{2}dx.
		\end{eqnarray*}
		By setting $\varepsilon_{1}=\frac{\alpha}{4N_{1}}$, and we choose $N_{1}$ large enough so that%
		\begin{equation*}
			\alpha_{1}=\kappa N_{1}-\rho_{2} >0.
		\end{equation*}
		Thus, we arrive at
		\begin{eqnarray}
		\mathcal{L}^{\prime }\left( t\right) &\leq & -\left[ N\mu-\rho_{1}\right] \int_{0}^{1}\varphi_{t}^{2}dx
		-\left[\frac{\mu\rho_{2}}{\kappa}N+\alpha_{2}-c\right] \int_{0}^{1}\varphi _{tt}^{2}dx\notag\\
		&&-\alpha_{3}\int_{0}^{1}\psi_{x}^{2} dx- \alpha_{4}
		\int_{0}^{1}(\varphi_{x}+\psi)^{2}dx- \alpha_{1}\int_{0}^{1}\varphi_{tx}^{2}dx\notag\\
		&&-\left[ Nh-c\right] \int_{0}^{1}P_{x}^{2}dx-\left[ N\delta-c\right] \int_{0}^{1}\theta_{x}^{2}dx, \label{eq2.28}
		\end{eqnarray}
		where $\alpha_{2}=\frac{\rho_{1}\rho_{2}}{\kappa}, \alpha_{3}=\frac{\alpha}{4}, \alpha_{4}=\kappa$.\\
		On the other hand, if we let
		\begin{equation*}
			\mathcal{H}\left( t\right) =N_{1}F_{1}(t)+F_{2}(t).
		\end{equation*}
		Then
		\begin{eqnarray*}
			\left\vert \mathcal{H}\left( t\right) \right\vert &\leq
			&\frac{\mu}{2}N_{1} \int_{0}^{1}\varphi_{t}^{2} dx
			+\kappa N_{1}\int_{0}^{1}\vert\varphi _{tx}\varphi_{x}\vert dx+\rho_{1}\int_{0}^{1}\vert\varphi\varphi_{t}\vert dx\\
			&&+\frac{\mu}{2}\int_{0}^{1}\varphi^{2} dx
			+\frac{\mu\rho_{2}}{2\kappa}\int_{0}^{1}\varphi_{t}^{2} dx+\rho_{2}\int_{0}^{1}\vert\varphi_{tx}\varphi_{x}\vert dx.
		\end{eqnarray*}
		By using Young, Poincar\'{e} inequalities and the fact that
		\begin{equation*}
			\int_{0}^{1}\varphi^{2} dx\leq 2c\int_{0}^{1}(\varphi_{x}+\psi)^{2} dx+2c\int_{0}^{1}\psi_{x}^{2} dx,
		\end{equation*}
		we get
		\begin{eqnarray*}
			\left\vert \mathcal{H}\left( t\right) \right\vert &\leq & c\int_{0}^{1}\left(
			\varphi_{t}^{2}+\varphi _{tx}^{2}+\psi _{x}^{2}+\varphi_{tt}^{2}+(\varphi_{x}+\psi) ^{2}\right) dx\\
			&\leq &c\mathcal{E}\left( t\right).
		\end{eqnarray*}
		Consequently,
		\begin{equation*}
			\left\vert \mathcal{H} \left( t\right) \right\vert =\left\vert \mathcal{L}\left( t\right) -N\mathcal{E}\left( t\right) \right\vert \leq
			c\mathcal{E}\left( t\right),
		\end{equation*}
		which yield
		\begin{equation}
		\left( N-c\right) \mathcal{E}\left( t\right) \leq \mathcal{L}\left( t\right) \leq
		\left( N+c\right)\mathcal{E}\left( t\right).  \label{eq2.29}
		\end{equation}
		Now, we choose $N$ large enough so that
		\begin{equation*}
			\frac{\mu\rho_{2}}{\kappa}N+\alpha_{2}-c>0, N\mu-\rho_{1}>0, N-c>0, N\delta-\rho_{1}>0, Nh-\rho_{1}>0.
		\end{equation*}
		We get
		\begin{equation}
		c_{1}\mathcal{E}(t) \leq \mathcal{L}\left( t\right) \leq c_{2}\mathcal{E}(t), \forall t\geq 0,  \label{eq2.31}
		\end{equation}
		and used (\ref{eq2.2}), estimates (\ref{eq2.28}), (\ref{eq2.29}),
		respectively, we get
		\begin{equation}
		\mathcal{L}^{\prime }\left( t\right) \leq -h_{1}\mathcal{E}\left( t\right), \forall t\geq t_{0},  \label{eq2.30}
		\end{equation}
		for some $h_{1}, c_{1}, c_{2}>0.$\\
		A combination (\ref{eq2.30}) with (\ref{eq2.31}), gives
		\begin{equation}
		\mathcal{L}^{\prime }\left( t\right) \leq -\lambda_{1}\mathcal{L}\left( t\right),  \label{eq2.32}
		\end{equation}
		where $\lambda_{1}=\frac{h_{1}}{c_{2}}$.
		Finally, a simple integration of (\ref{eq2.32}) we obtain (\ref{eq2.26}). This completes the proof.
	\end{proof}

\section{Numerical approximation}
To obtain the week formulation we multiply (\ref{sys1.15}) by $\bar{\varphi},$ $\bar{\psi}$,  $\bar{\theta}$ and $\bar{P}$, then integrating by part where $\Phi =\varphi
_{t}$ and $\Psi =\psi _{t}$ to obtain

\begin{eqnarray}
\left\{
\begin{array}{ll}
\rho _{1}(\Phi _{t},\bar{\varphi})+\kappa (\varphi _{x}+\psi,\bar{\varphi}%
_{x})=0 \\
\rho _{2}(\Phi _{t},\bar{\psi}_{x})+\alpha (\psi _{x}, \bar{\psi}_{x})+\kappa
(\varphi _{x}+\psi,\bar{\psi})+\xi _{1}(\theta, \bar{\psi}_{x})+\xi _{2}(P,%
\bar{\psi}_{x})=0 \\
c(\theta _{t}, \bar{\theta})+d(P_{t},\bar{\theta})+\kappa (\theta _{x},\bar{%
\theta}_{x})+\xi _{1}(\Psi,\bar{\theta}_{x})=0 \\
d(\theta _{t}, \bar{P})+r(P_{t}, \bar{P})+h(P_{x}, \bar{P}_{x})+\xi _{2}(\Psi,%
\bar{P}_{x})=0,
\end{array}
\right.   \label{1}
\end{eqnarray}
where $(.,.)$ is the inner product in $L^{2}(0, L)$.\\
Let us partition the interval $(0, L)$ into subintervals $I_{j}=(x_{j-1}, x_{j})$ of length $h=\frac{1}{s}$ with $0=x_{0}<x_{1}<\ldots <x_{s}=L$ and define the linear \ polynomial
$$
S_{0}^{h}=\{u\in H_{0}^{1}(0, L)|u\in C([0, L]), u|_{I_{j}}\in P_1(K)\}.
$$
For a given final time $T$ and a positive integer $N$, let $\Delta t=\frac{T}{N}$ be the time step and $t_{n}=n\Delta t, n=0, \ldots, N$. The finite
element method for the Dirichlet homogeneous boundary condition is to $\Phi
_{h}^{n},$ $\Psi _{h}^{n},$ $\theta _{h}^{n},$ $P_{h}^{n}\in S_{0}^{h}$, $%
n=1, \ldots, N$, such that for all $\bar{\varphi}_{h}$, $\bar{\psi}_{h}$, $%
\bar{\theta}_{h},$ $\bar{P}_{h}$, we have
\begin{eqnarray}
\left\{
\begin{array}{ll}
\frac{\rho _{1}}{\Delta t}(\Phi _{h}^{n}-\Phi _{h}^{n-1},\bar{\varphi}%
_{h})+\kappa (\varphi _{hx}^{n}+\psi _{h}^{n},\bar{\varphi}_{hx})=0 \\
\frac{\rho _{2}}{\Delta t}(\Phi _{h}^{n}-\Phi _{h}^{n-1},\bar{\psi}%
_{h, x})+\alpha (\psi _{hx}^{n},\bar{\psi}_{hx})+\kappa (\varphi
_{hx}^{n}+\psi _{h}^{n},\bar{\psi}_{h})\\
\ \ \ \ \ \ \ \ \ \ \ \ \ \ \ \ \ \ \ \ \ \ \ \ \ \ \ \ \ \ +\xi _{1}(\theta _{h}^{n},\bar{\psi}%
_{hx})+\xi _{2}(P_{h}^{n},\bar{\psi}_{hx})=0 \\
\frac{c}{\Delta t}(\theta _{h}^{n}-\theta _{h}^{n-1},\bar{\theta}_{h})+\frac{%
	d}{\Delta t}(P_{h}^{n}-P_{h}^{n-1},\bar{\theta}_{h})+\kappa (\theta
_{hx}^{n},\bar{\theta}_{hx})+\xi _{1}(\Psi _{h}^{n},\bar{\theta}_{hx})=0 \\
\frac{d}{\Delta t}(\theta _{h}^{n}-\theta _{h}^{n-1},\bar{P}_{h})+\frac{r}{%
	\Delta t}(P_{h}^{n}-P_{h}^{n-1},\bar{P}_{h})+h(P_{hx}^{n},\bar{P}_{hx})+\xi
_{2}(\Psi _{h}^{n},\bar{P}_{hx})=0,
\end{array}%
\right.  \label{4}
\end{eqnarray}
where $\varphi _{h}^{n}=\varphi _{h}^{n-1}+\Delta t\Phi _{h}^{n}$ and $\psi
_{h}^{n}=\psi _{h}^{n-1}+\Delta t\Psi _{h}^{n}$. \\
Here, $\varphi _{h}^{0}$, $%
\Phi _{h}^{0}$, $\psi _{h}^{0}$, $\Psi _{h}^{0}$, $\theta _{h}^{0}$, $%
P_{h}^{0}$ are approximations to $\varphi ^{0}$, $\varphi ^{1}$, $\psi ^{0}$, $\psi ^{1}$, $\theta ^{0}$, $P^{0}$ respectively.\\
 Let us introduce the
discrete energy
\begin{eqnarray}
E^{n}=\frac{1}{2}\Big(\rho _{1}\left\vert \left\vert \Phi _{h}^{n}\right\vert
\right\vert ^{2}+\frac{\rho _{1}\rho _{2}}{k}\left\vert \left\vert \bar{%
	\partial}\Phi _{h}^{n}\right\vert \right\vert ^{2}+\rho _{2}\left\vert
\left\vert \Phi _{h, x}^{n}\right\vert \right\vert ^{2}+\kappa \left\vert
\left\vert \varphi _{hx}^{n}+\psi _{h}^{n}\right\vert \right\vert
^{2}\nonumber\\
+r\left\vert \left\vert P_{h}^{n}\right\vert \right\vert
^{2}+c\left\vert \left\vert \theta _{h}^{n}\right\vert \right\vert
^{2}+\alpha \left\vert \left\vert \psi _{hx}^{n}\right\vert \right\vert
^{2}+2d(P_{h}^{n},\theta _{h}^{n})\Big). \label{Eh}
\end{eqnarray}
\begin{theorem}\label{Th4}
	The discrete energy (\ref{Eh}) decays to zero as $t$ goes to $\infty$, that is,
\end{theorem}
\[
\frac{E^{n}-E^{n-1}}{\Delta t}\leq 0, n=1, \ldots, N
\]
\begin{proof}
	Taking
	$\bar{\varphi}_h=\widehat{\varphi}_h^n$, $\bar{\psi}_h=\widehat{\psi}_h^n$,  $\bar{\theta}_h=\theta_h^n$, and $\bar{P}_h=P_h^n$ in the scheme to get
\begin{equation}	\label{e1}
	\left\{ \begin{array}{ll}
	\frac{\rho _{1}}{2\Delta t}(\left\vert \left\vert \Phi _{h}^{n}-\Phi
	_{h}^{n-1}\right\vert \right\vert ^{2}+\left\vert \left\vert \Phi
	_{h}^{n}\right\vert \right\vert ^{2}-\left\vert \left\vert \Phi
	_{h}^{n-1}\right\vert \right\vert ^{2})+\kappa (\varphi _{hx}^{n}+\psi
	_{h}^{n},\Phi _{hx}^{n})=0 \\
	\frac{\rho _{2}}{\Delta t}(\Phi _{h}^{n}-\Phi _{h}^{n-1},\bar{\psi}%
	_{h, x})+\alpha (\psi _{hx}^{n},\Psi _{hx}^{n})+\kappa (\varphi
	_{hx}^{n}+\psi _{h}^{n},\Psi _{h}^{n})+\xi _{1}(\theta _{h}^{n},\Psi
	_{hx}^{n})\\
	\qquad \qquad\qquad \qquad\qquad+\xi _{2}(P_{h}^{n},\Psi _{hx}^{n})+\mu \left\vert \left\vert \Psi \right\vert \right\vert ^{2}=0  \\
	\frac{c}{2\Delta t}(\left\vert \left\vert \theta _{h}^{n}-\theta
	_{h}^{n-1}\right\vert \right\vert ^{2}+\left\vert \left\vert \theta
	_{h}^{n}\right\vert \right\vert ^{2}-\left\vert \left\vert \theta
	_{h}^{n-1}\right\vert \right\vert ^{2})+\frac{d}{\Delta t}%
	(P_{h}^{n}-P_{h}^{n-1},\theta _{h}^{n})\\
	\qquad +\kappa (\theta _{hx}^{n},\theta
	_{hx}^{n})+\xi _{1}(\Psi _{h}^{n},\theta _{hx}^{n})=0,\\
	\frac{r}{2\Delta t}(\left\vert \left\vert P_{h}^{n}-P_{h}^{n-1}\right\vert
	\right\vert ^{2}+\left\vert \left\vert P_{h}^{n}\right\vert \right\vert
	^{2}-\left\vert \left\vert P_{h}^{n-1}\right\vert \right\vert ^{2})+\frac{d}{%
		\Delta t}(\theta _{h}^{n}-\theta
	_{h}^{n-1}, P_{h}^{n})\\
	\qquad\qquad\qquad\qquad\qquad +h(P_{hx}^{n}, P_{hx}^{n})+\xi _{2}(\Psi
	_{h}^{n}, P_{hx}^{n})=0.
	\end{array} \right.
	\end{equation}
	Then by (\ref{e1}), we have
	\begin{eqnarray}
	0=&& \nonumber
	\frac{\rho _{1}}{2\Delta t}\Big(\left\vert \left\vert \Phi _{h}^{n}-\Phi
	_{h}^{n-1}\right\vert \right\vert ^{2}+\left\vert \left\vert \Phi
	_{h}^{n}\right\vert \right\vert ^{2}-\left\vert \left\vert \Phi
	_{h}^{n-1}\right\vert \right\vert ^{2}\Big)+\frac{\rho _{2}}{\Delta t}\Big(\Phi
	_{h}^{n}-\Phi _{h}^{n-1},\Psi _{h, x}^{n}\Big) \nonumber \\
	&&+\frac{c}{2\Delta t}\Big(\left\vert \left\vert \theta _{h}^{n}-\theta
	_{h}^{n-1}\right\vert \right\vert ^{2}+\left\vert \left\vert \theta
	_{h}^{n}\right\vert \right\vert ^{2}-\left\vert \left\vert \theta
	_{h}^{n-1}\right\vert \right\vert ^{2}\Big)\nonumber\\
	&&+\frac{r}{2\Delta t}\Big(\left\vert
	\left\vert P_{h}^{n}-P_{h}^{n-1}\right\vert \right\vert ^{2}+\left\vert
	\left\vert P_{h}^{n}\right\vert \right\vert ^{2}-\left\vert \left\vert
	P_{h}^{n-1}\right\vert \right\vert ^{2}\Big)\nonumber\\
	&& 	+\frac{d}{\Delta t}(P_{h}^{n}-P_{h}^{n-1},\theta _{h}^{n})+\kappa
	\left\vert \left\vert \theta _{hx}^{n}\right\vert \right\vert ^{2}+\frac{d}{%
		\Delta t}(\theta _{h}^{n}-\theta _{h}^{n-1}, P_{h}^{n})+h\left\vert
	\left\vert P_{hx}^{n}\right\vert \right\vert ^{2}\nonumber\\
	&& +\alpha (\psi _{hx}^{n},\Psi _{hx}^{n})+\kappa (\varphi _{hx}^{n}+\psi
	_{h}^{n},\Phi _{hx}^{n}+\Psi _{h}^{n})+\mu \left\vert \left\vert \Psi \right\vert \right\vert ^{2},\nonumber
	\end{eqnarray}
	by the fact that
	\begin{eqnarray}
	&& \nonumber
	\kappa (\varphi _{hx}^{n}+\psi _{h}^{n},\Phi _{hx}^{n}+\Psi _{h}^{n}) \geq \frac{\kappa }{2\Delta t}%
	\Big(\left\vert \left\vert \varphi _{hx}^{n}+\psi _{h}^{n}\right\vert
	\right\vert ^{2}-\left\vert \left\vert \varphi _{hx}^{n-1}+\psi
	_{h}^{n-1}\right\vert \right\vert ^{2}\Big),
	\end{eqnarray}
	and
	\begin{eqnarray}
	\alpha (\psi _{hx}^{n},\Psi _{hx}^{n})=\frac{\alpha }{\Delta t}(\psi
	_{hx}^{n},\psi _{hx}^{n}-\psi _{hx}^{n-1})\geq\frac{\alpha }{2\Delta t}%
	\Big(\left\vert \left\vert \psi _{hx}^{n}\right\vert \right\vert ^{2}-\left\vert
	\left\vert \psi _{hx}^{n-1}\right\vert \right\vert ^{2}\Big).
	\end{eqnarray}
	Moreover, from (\ref{sys1.15})$_1$, we have
	\[
	\psi _{xt}=\frac{\rho _{1}}{k}(\varphi _{tt})_{t}-\varphi _{txx},
	\]
	\\
	then,
	\begin{eqnarray}
	&& (\Phi _{h}^{n}-\Phi _{h}^{n-1},\Psi _{h, x}^{n})\nonumber\\	
	&&=(\Phi _{h}^{n}-\Phi _{h}^{n-1},\frac{\rho _{1}}{k}\Phi _{tt}-\Phi _{xx})\nonumber\\
	&&=\frac{\rho _{1}}{k\Delta t^{2}}\big(\Phi _{h}^{n}-\Phi _{h}^{n-1},\Phi
	_{h}^{n}-\Phi _{h}^{n-1}-(\Phi _{h}^{n-1}-\Phi _{h}^{n-2})\big)+(\Phi
	_{h, x}^{n}-\Phi _{h, x}^{n-1},\Phi _{h, x}^{n})\nonumber\\
	&&\geq\frac{\rho _{1}}{2k\Delta t^{2}}\Big(\left\vert \left\vert \Phi
	_{h}^{n}-\Phi _{h}^{n-1}\right\vert \right\vert ^{2}-\left\vert \left\vert
	\Phi _{h}^{n-1}-\Phi _{h}^{n-2}\right\vert \right\vert ^{2}\Big)+\frac{1}{2}\Big(\left\vert
	\left\vert \Phi _{h, x}^{n}\right\vert \right\vert ^{2}-\left\vert \left\vert
	\Phi _{h, x}^{n-1}\right\vert \right\vert ^{2}\Big).\nonumber
	\end{eqnarray}
	Then we get,
	\begin{eqnarray}
	&&\frac{\rho _{1}}{2\Delta t}\Big(\left\vert \left\vert \Phi _{h}^{n}-\Phi
	_{h}^{n-1}\right\vert \right\vert ^{2}+\left\vert \left\vert \Phi
	_{h}^{n}\right\vert \right\vert ^{2}-\left\vert \left\vert \Phi
	_{h}^{n-1}\right\vert \right\vert ^{2}\Big)+\frac{\rho _{2}}{\Delta t}\Big(\Phi
	_{h}^{n}-\Phi _{h}^{n-1}, \Psi _{h, x}^{n}\Big) \nonumber\\
	&&+\frac{c}{2\Delta t}\Big(\left\vert \left\vert \theta _{h}^{n}-\theta
	_{h}^{n-1}\right\vert \right\vert ^{2}+\left\vert \left\vert \theta
	_{h}^{n}\right\vert \right\vert ^{2}-\left\vert \left\vert \theta
	_{h}^{n-1}\right\vert \right\vert ^{2}\Big)\nonumber\\
	&&+\frac{r}{2\Delta t}\Big(\left\vert
	\left\vert P_{h}^{n}-P_{h}^{n-1}\right\vert \right\vert ^{2}+\left\vert
	\left\vert P_{h}^{n}\right\vert \right\vert ^{2}-\left\vert \left\vert
	P_{h}^{n-1}\right\vert \right\vert ^{2}\Big)\nonumber\\
	&&+\frac{d}{\Delta t}(P_{h}^{n}-P_{h}^{n-1},\theta _{h}^{n})+\kappa
	\left\vert \left\vert \theta _{hx}^{n}\right\vert \right\vert ^{2}+\frac{d}{%
		\Delta t}(\theta _{h}^{n}-\theta _{h}^{n-1}, P_{h}^{n})+h\left\vert
	\left\vert P_{hx}^{n}\right\vert \right\vert ^{2}\nonumber\\
	&&+\alpha (\psi _{hx}^{n},\Psi _{hx}^{n})+\kappa (\varphi _{hx}^{n}+\psi
	_{h}^{n},\Phi _{hx}^{n}+\Psi _{h}^{n})+\mu \left\vert \left\vert \Psi \right\vert \right\vert ^{2}\nonumber\\
	&&\geq	\frac{\rho _{1}}{2\Delta t}\Big(\left\vert \left\vert \Phi _{h}^{n}-\Phi
	_{h}^{n-1}\right\vert \right\vert ^{2}+\left\vert \left\vert \Phi
	_{h}^{n}\right\vert \right\vert ^{2}-\left\vert \left\vert \Phi
	_{h}^{n-1}\right\vert \right\vert ^{2}\Big)\nonumber \\
	&&+\frac{\rho _{1}\rho _{2}}{2k\Delta t^{3}}\Big(\left\vert \left\vert \Phi
	_{h}^{n}-\Phi _{h}^{n-1}\right\vert \right\vert ^{2}-\left\vert \left\vert
	\Phi _{h}^{n-1}-\Phi _{h}^{n-2}\right\vert \right\vert ^{2}\Big)\nonumber\\
	&&+\frac{\rho _{1}}{2\Delta t}\Big(\left\vert
	\left\vert \Phi _{h, x}^{n}\right\vert \right\vert ^{2}-\left\vert \left\vert
	\Phi _{h, x}^{n-1}\right\vert \right\vert ^{2}\Big) \nonumber\\
	&&+\frac{c}{2\Delta t}\Big(\left\vert \left\vert \theta _{h}^{n}-\theta
	_{h}^{n-1}\right\vert \right\vert ^{2}+\left\vert \left\vert \theta
	_{h}^{n}\right\vert \right\vert ^{2}-\left\vert \left\vert \theta
	_{h}^{n-1}\right\vert \right\vert ^{2}\Big)
	\nonumber\\
	&&+\frac{r}{2\Delta t}\Big(\left\vert
	\left\vert P_{h}^{n}-P_{h}^{n-1}\right\vert \right\vert ^{2}+\left\vert
	\left\vert P_{h}^{n}\right\vert \right\vert ^{2}-\left\vert \left\vert
	P_{h}^{n-1}\right\vert \right\vert ^{2}\Big)\nonumber\\
	&&+\frac{d}{\Delta t}(P_{h}^{n}-P_{h}^{n-1},\theta _{h}^{n})+\kappa
	\left\vert \left\vert \theta _{hx}^{n}\right\vert \right\vert ^{2}+\frac{d}{%
		\Delta t}(\theta _{h}^{n}-\theta _{h}^{n-1}, P_{h}^{n})+h\left\vert
	\left\vert P_{hx}^{n}\right\vert \right\vert ^{2}\nonumber\\
	&&+\alpha (\psi _{hx}^{n},\Psi _{hx}^{n})+\kappa (\varphi _{hx}^{n}+\psi
	_{h}^{n},\Phi _{hx}^{n}+\Psi _{h}^{n})+\mu \left\vert \left\vert \Psi \right\vert \right\vert ^{2}.\nonumber
	\end{eqnarray}
	This lead to
	\begin{eqnarray}
 0\geq	&&  \frac{\rho _{1}}{2\Delta t}\Big(\left\vert \left\vert \Phi _{h}^{n}-\Phi
	_{h}^{n-1}\right\vert \right\vert ^{2}+\left\vert \left\vert \Phi
	_{h}^{n}\right\vert \right\vert ^{2}-\left\vert \left\vert \Phi
	_{h}^{n-1}\right\vert \right\vert ^{2}\Big) 	\nonumber \\
	&&+\frac{\rho _{1}\rho _{2}}{2k\Delta t^{3}}\Big(\left\vert \left\vert \Phi
	_{h}^{n}-\Phi _{h}^{n-1}\right\vert \right\vert ^{2}-\left\vert \left\vert
	\Phi _{h}^{n-1}-\Phi _{h}^{n-2}\right\vert \right\vert ^{2}\Big)\nonumber\\
	&&+\frac{\rho _{1}}{2\Delta t}\Big(\left\vert
	\left\vert \Phi _{h, x}^{n}\right\vert \right\vert ^{2}-\left\vert \left\vert
	\Phi _{h, x}^{n-1}\right\vert \right\vert ^{2}\Big)\nonumber \\
	&&
	+\frac{c}{2\Delta t}\Big(\left\vert \left\vert \theta _{h}^{n}-\theta
	_{h}^{n-1}\right\vert \right\vert ^{2}+\left\vert \left\vert \theta
	_{h}^{n}\right\vert \right\vert ^{2}-\left\vert \left\vert \theta
	_{h}^{n-1}\right\vert \right\vert ^{2}\Big)
		\nonumber\\
	&&	+\frac{r}{2\Delta t}\Big(\left\vert
	\left\vert P_{h}^{n}-P_{h}^{n-1}\right\vert \right\vert ^{2}+\left\vert
	\left\vert P_{h}^{n}\right\vert \right\vert ^{2}-\left\vert \left\vert
	P_{h}^{n-1}\right\vert \right\vert ^{2}\Big) \nonumber\\
	&&
	+\frac{d}{\Delta t}(P_{h}^{n}-P_{h}^{n-1},\theta _{h}^{n})+\kappa
	\left\vert \left\vert \theta _{hx}^{n}\right\vert \right\vert ^{2}+\frac{d}{%
		\Delta t}(\theta _{h}^{n}-\theta _{h}^{n-1}, P_{h}^{n})+h\left\vert
	\left\vert P_{hx}^{n}\right\vert \right\vert ^{2}\nonumber\\
	&& +\frac{\alpha }{2\Delta t}%
	\Big(\left\vert \left\vert \psi _{hx}^{n}\right\vert \right\vert ^{2}-\left\vert
	\left\vert \psi _{hx}^{n-1}\right\vert \right\vert ^{2}\Big)\nonumber\\
	&&+\frac{\kappa }{2\Delta t}%
	\Big(\left\vert \left\vert \varphi _{hx}^{n}+\psi _{h}^{n}\right\vert
	\right\vert ^{2}-\left\vert \left\vert \varphi _{hx}^{n-1}+\psi
	_{h}^{n-1}\right\vert \right\vert ^{2}\Big)+\mu \left\vert \left\vert \Psi \right\vert \right\vert ^{2}.\nonumber
	\end{eqnarray}
	Take into the consideration that,
	\begin{eqnarray}
	 &&\frac{d}{\Delta t}(P_{h}^{n}-P_{h}^{n-1},\theta _{h}^{n})+\frac{d}{\Delta t}%
	(\theta _{h}^{n}-\theta _{h}^{n-1}, P_{h}^{n})\nonumber\\
	&&=\frac{d}{\Delta t}\Big((P_{h}^{n},\theta _{h}^{n})-(P_{h}^{n-1},\theta
	_{h}^{n-1})+(\theta _{h}^{n}-\theta _{h}^{n-1}, P_{h}^{n}-P_{h}^{n-1})\Big), \nonumber
	\end{eqnarray}
	and since
		\begin{eqnarray}
	\frac{c}{2\Delta t}\left\vert \left\vert \theta _{h}^{n}-\theta
	_{h}^{n-1}\right\vert \right\vert ^{2} +\frac{r}{2\Delta t}\left\vert
	\left\vert P_{h}^{n}-P_{h}^{n-1}\right\vert \right\vert ^{2}+\frac{d}{\Delta
		t}(\theta _{h}^{n}-\theta _{h}^{n-1}, P_{h}^{n}-P_{h}^{n-1})>0.\nonumber
	\end{eqnarray}
	Then,
	\begin{eqnarray}
	0\geq &&\frac{\rho _{1}}{2\Delta t}\Big(\left\vert \left\vert \Phi
	_{h}^{n}\right\vert \right\vert ^{2}-\left\vert \left\vert \Phi
	_{h}^{n-1}\right\vert \right\vert ^{2}\Big)+\frac{c}{2\Delta t}\Big(\left\vert \left\vert \theta
	_{h}^{n}\right\vert \right\vert ^{2}-\left\vert \left\vert \theta
	_{h}^{n-1}\right\vert \right\vert ^{2}\Big) 	\nonumber\\ 	
	&&+\frac{\rho _{1}\rho _{2}}{2k\Delta t^{3}}\Big(\left\vert \left\vert \Phi
	_{h}^{n}-\Phi _{h}^{n-1}\right\vert \right\vert ^{2}-\left\vert \left\vert
	\Phi _{h}^{n-1}-\Phi _{h}^{n-2}\right\vert \right\vert ^{2}\Big)+\frac{\rho _{1}}{2\Delta t}\Big(\left\vert
	\left\vert \Phi _{h, x}^{n}\right\vert \right\vert ^{2}-\left\vert \left\vert
	\Phi _{h, x}^{n-1}\right\vert \right\vert ^{2}\Big) 	\nonumber \\
	&&
	+\frac{r}{2\Delta t}\Big(\left\vert
	\left\vert P_{h}^{n}\right\vert \right\vert ^{2}-\left\vert \left\vert
	P_{h}^{n-1}\right\vert \right\vert ^{2}\Big)
	+\frac{d}{\Delta t}\Big((P_{h}^{n}, \theta _{h}^{n})-(P_{h}^{n-1},\theta
	_{h}^{n-1})\Big)\nonumber\\
	&&
 +\frac{\alpha }{2\Delta t}%
	\Big(\left\vert \left\vert \psi _{hx}^{n}\right\vert \right\vert ^{2}-\left\vert
	\left\vert \psi _{hx}^{n-1}\right\vert \right\vert ^{2}\Big)+\frac{\kappa }{2\Delta t}%
	\Big(\left\vert \left\vert \varphi _{hx}^{n}+\psi _{h}^{n}\right\vert
	\right\vert ^{2}-\left\vert \left\vert \varphi _{hx}^{n-1}+\psi
	_{h}^{n-1}\right\vert \right\vert ^{2}\Big), \nonumber
	\end{eqnarray}
	this completes the proof.
\end{proof}
\section{Priori error estimate}
In this section we obtain a priori error estimate on the numerical approximations, in which we obtain the convergence of the error.
\begin{theorem}
	Suppose that the solution $(\varphi, \psi, \theta, P)$ of (\ref{4}) belong to the space
	\[
	\left( H^4(0, T, H^{2}(0, L))\right) ^{4}.
	\]
	Then the following priori error estimate holds
	\begin{eqnarray}
	&&\left\vert \left\vert \Phi _{h}^{n}-\varphi _{t}(t_{n})\right\vert
	\right\vert ^{2}+\left\vert \left\vert \varphi _{h, x}^{n}-(\varphi
	(t_{n}))_{x}+\psi _{h}^{n}-\psi (t_{n})\right\vert \right\vert
	^{2}+\left\vert \left\vert \Psi _{h}^{n}-\psi _{t}(t_{n})\right\vert
	\right\vert ^{2}\nonumber\\
	&&+\left\vert \left\vert \psi _{h, x}^{n}-(\psi
	(t_{n}))_{x}\right\vert \right\vert ^{2}+\left\vert \left\vert \theta
	_{h}^{n}-\theta (t_{n})\right\vert \right\vert ^{2}+\left\vert \left\vert
	P_{h}^{n}-P(t_{n})\right\vert \right\vert ^{2}<c(\Delta t^{2}+h^{2}).\nonumber
	\end{eqnarray}
\end{theorem}
\begin{proof}
	Let
		\begin{eqnarray}
	e^n=\varphi^n_h-P^0_h\varphi(t_n)  \nonumber
	\end{eqnarray}
	\begin{eqnarray}
	\widehat{e}_n=\widehat{\varphi}^n_h-P^0_h\varphi_t(t_n)  \nonumber
		\end{eqnarray}
			\begin{eqnarray}
	q^n=\psi^n_h-P^0_h\psi(t_n)  \nonumber
	\end{eqnarray}
	\begin{eqnarray}
	\widehat{q}_n=\widehat{\psi}_h^{n}-P^0_h\psi_t(t_n)  \nonumber
	\end{eqnarray}
	\begin{eqnarray}
	R^n=\theta _{h}^{n}-P_{h}^{0}\theta (t_{n})  \nonumber
	\end{eqnarray}
	\begin{eqnarray}
	J^{n}=P_{h}^{n}-P_{h}^{0}P(t_{n}). \nonumber
		\end{eqnarray}
	\textbf{Step 1: }Substitute in the scheme taking $\bar{\varphi}_{h}=\widehat{e}%
	^{n}$, $\bar{\psi}_{h}=\widehat{q}^{n},$ $\bar{\theta}_{h}=R^{n},$ $\bar{P}%
	_{h}=J^{n}.$
	\begin{eqnarray}
	\left\{
	\begin{array}{l}
	\frac{\rho _{1}}{\Delta t}(\widehat{e}^{n}+P_{h}^{0}\varphi _{t}(t_{n})-(\widehat{e}%
	^{n-1}+P_{h}^{0}\varphi _{t}(t_{n-1})), \widehat{e}^{n})+\kappa
	(e_{x}^{n}+(P_{h}^{0}\varphi (t_{n}))_{x}+q^{n}+P_{h}^{0}\psi (t_{n}), \widehat{e}%
	_{x}^{n})=0 \\
	\\
	\frac{\rho _{2}}{\Delta t}(\widehat{e}^{n}+P_{h}^{0}\varphi _{t}(t_{n})-(\widehat{e}%
	^{n-1}+P_{h}^{0}\varphi _{t}(t_{n-1})), \widehat{q}_{x}^{n})+\alpha
	(q_{x}^{n}+(P_{h}^{0}\psi (t_{n}))_{x}, \widehat{q}_{x}^{n})+ \\
	\kappa (e_{x}^{n}+(P_{h}^{0}\varphi (t_{n}))_{x}+q^{n}+P_{h}^{0}\psi (t_{n}),%
	\widehat{q}^{n})+\xi _{1}(R^{n}+P_{h}^{0}\theta (t_{n}), \widehat{q}_{x}^{n})+\xi
	_{2}(J^{n}+P_{h}^{0}P(t_{n}), \widehat{q}_{x}^{n})\\+\frac{\mu }{\Delta t}%
	(q^{n}+P_{h}^{0}\psi (t_{n})-(q^{n-1}+P_{h}^{0}\psi (t_{n-1})), \widehat{q}^{n})=0
	\\
	\\
	\frac{c}{\Delta t}(R^{n}+P_{h}^{0}\theta (t_{n})-(R^{n-1}+P_{h}^{0}\theta
	(t_{n-1})), R^{n})+\frac{d}{\Delta t}
	(J^{n}+P_{h}^{0}P(t_{n})-(J^{n-1}+P_{h}^{0}P(t_{n-1})), R^{n}) \\
	+\kappa (R_{x}^{n}+(P_{h}^{0}\theta (t_{n}))_{x}, R_{x}^{n})+\xi _{1}(\widehat{q}
	^{n}+P_{h}^{0}\psi _{t}(t_{n}), R_{x}^{n})=0 \\
	\\
	\frac{d}{\Delta t}(R^{n}+P_{h}^{0}\theta (t_{n})-(R^{n-1}+P_{h}^{0}\theta
	(t_{n-1})), J^{n})+\frac{r}{\Delta t}
	(J^{n}+P_{h}^{0}P(t_{n})-(J^{n-1}+P_{h}^{0}P(t_{n-1})), J^{n}) \\
	+h(J_{x}^{n}+(P_{h}^{0}P(t_{n}))_{x}, J_{x}^{n})+\xi _{2}(\widehat{q}
	^{n}+P_{h}^{0}\psi _{t}(t_{n}), J_{x}^{n})=0.
	\end{array}
	\right.
	\end{eqnarray}
	Take into consideration that
	\[
	\widehat{q}_{x}^{n}=\frac{\rho _{1}}{k}\delta ^{2}\widehat{e}^{n}-\widehat{e}_{xx}^{n},
	\]
	then
	\begin{eqnarray}
		&&(\widehat{e}^{n}-\widehat{e}^{n-1}, \frac{\rho _{1}}{k}\delta ^{2}\widehat{e}^{n}-\widehat{e%
		}_{xx}^{n}) \nonumber\\
		&=&\frac{\rho _{1}}{\Delta tk}(\widehat{e}^{n}-\widehat{e}^{n-1}, \delta (\widehat{e}^{n}-%
		\widehat{e}^{n-1}))+(\widehat{e}_{x}^{n}-\widehat{e}_{x}^{n-1}, \widehat{e}_{x}^{n}) \nonumber\\
		&=&\frac{\rho _{1}}{\Delta t^{2}k}(\left\vert \left\vert \widehat{e}^{n}-\widehat{e}%
		^{n-1}\right\vert \right\vert ^{2}-\left\vert \left\vert \widehat{e}^{n-1}-\widehat{e%
		}^{n-2}\right\vert \right\vert ^{2}-\Vert \widehat{e}^{n}-\widehat{e}^{n-1}-(\widehat{e}%
		^{n-1}-\widehat{e}^{n-2})\Vert^{2})+(\widehat{e}_{x}^{n}-\widehat{e}_{x}^{n-1}, \widehat{e}%
		_{x}^{n}) \nonumber \\
		&=&\frac{\rho _{1}}{2k}(\left\vert \left\vert \delta \widehat{e}^{n}\right\vert
		\right\vert ^{2}-\left\vert \left\vert \delta \widehat{e}^{n-1}\right\vert
		\right\vert ^{2}+\left\vert \left\vert \delta (\widehat{e}^{n}-\widehat{e}%
		^{n-1})\right\vert \right\vert ^{2})+\frac{1}{2}(\left\vert \left\vert \widehat{e}%
		_{x}^{n}\right\vert \right\vert ^{2}-\left\vert \left\vert \widehat{e}%
		_{x}^{n-1}\right\vert \right\vert ^{2}+\left\vert \left\vert \widehat{e}_{x}^{n}-%
		\widehat{e}_{x}^{n-1}\right\vert \right\vert ^{2}), \nonumber
	\end{eqnarray}
	and
	\begin{eqnarray}
		&&(P_{h}^{0}\varphi _{t}(t_{n})-P_{h}^{0}\varphi _{t}(t_{n-1}), \widehat{q}%
		_{x}^{n}) \nonumber\\
		&=&(P_{h}^{0}\varphi _{t}(t_{n})-P_{h}^{0}\varphi _{t}(t_{n-1}), \frac{\rho
			_{1}}{k}\delta ^{2}\widehat{e}^{n}-\widehat{e}_{xx}^{n}) \nonumber\\
		&=&(P_{h}^{0}\varphi _{t}(t_{n})-P_{h}^{0}\varphi _{t}(t_{n-1}), \frac{\rho
			_{1}}{k}\delta ^{2}\widehat{e}^{n}-\widehat{e}_{xx}^{n}) \nonumber\\
		&=&(P_{h}^{0}\varphi _{t}(t_{n})-P_{h}^{0}\varphi _{t}(t_{n-1}), \frac{\rho
			_{1}}{k}\delta ^{2}\widehat{e}^{n})+((P_{h}^{0}\varphi
		_{t}(t_{n}))_{x}-(P_{h}^{0}\varphi _{t}(t_{n-1}))_{x}, \widehat{e}_{x}^{n}).\nonumber
	\end{eqnarray}
	Then
	\begin{eqnarray}
	\left\{
	\begin{array}{ll}
  \frac{\rho _{1}}{\Delta t}(\left\vert \left\vert \widehat{e}^{n}-\widehat{e}
	^{n-1}\right\vert \right\vert ^{2}+\left\vert \left\vert \widehat{e}
	^{n}\right\vert \right\vert ^{2}-\left\vert \left\vert \widehat{e}
	^{n-1}\right\vert \right\vert ^{2})+\frac{\rho _{1}}{\Delta t}
	(P_{h}^{0}\varphi _{t}(t_{n})-P_{h}^{0}\varphi _{t}(t_{n-1}), e^{n}) \\
 + \kappa ((P_{h}^{0}\varphi (t_{n}))_{x}+P_{h}^{0}\psi (t_{n}), \widehat{e}
	_{x}^{n})+\kappa (e_{x}^{n}+q^{n}, \widehat{e}_{x}^{n})=0 \\ \\
  \frac{\rho _{2}}{\Delta t}(\frac{\rho _{1}}{2k}(\left\vert \left\vert \delta
	\widehat{e}^{n}\right\vert \right\vert ^{2}-\left\vert \left\vert \delta \widehat{e}
	^{n-1}\right\vert \right\vert ^{2}+\left\vert \left\vert \delta (\widehat{e}^{n}-
	\widehat{e}^{n-1})\right\vert \right\vert ^{2})+\frac{1}{2}(\left\vert \left\vert \widehat{e}
	_{x}^{n}\right\vert \right\vert ^{2}-\left\vert \left\vert \widehat{e}
	_{x}^{n-1}\right\vert \right\vert ^{2}+\left\vert \left\vert \widehat{e}_{x}^{n}-
	\widehat{e}_{x}^{n-1}\right\vert \right\vert ^{2}))  \\
	  +\frac{\rho _{2}}{\Delta t}(P_{h}^{0}\varphi _{t}(t_{n})-P_{h}^{0}\varphi
	_{t}(t_{n-1}), \frac{\rho _{1}}{k}\delta ^{2}\widehat{e}^{n}) +\frac{\rho _{2}}{
		\Delta t}((P_{h}^{0}\varphi _{t}(t_{n}))_{x}-(P_{h}^{0}\varphi
	_{t}(t_{n-1}))_{x}, \widehat{e}_{x}^{n})+\alpha ((P_{h}^{0}\psi (t_{n}))_{x}, \widehat{
		q}_{x}^{n})   \\
 	+\alpha (q_{x}^{n}, \widehat{q}_{x}^{n})+\kappa ((P_{h}^{0}\varphi
	(t_{n}))_{x}+P_{h}^{0}\psi (t_{n}), \widehat{q}^{n})+\kappa (e_{x}^{n}+q^{n}, \widehat{q}^{n}) +\xi _{1}(P_{h}^{0}\theta (t_{n}), \widehat{q}_{x}^{n})+\xi _{1}(R^{n},
	\widehat{q}_{x}^{n})   \\
 +\xi _{2}(P_{h}^{0}P(t_{n}), \widehat{q}_{x}^{n})+\xi _{2}(J^{n}, \widehat{q}
	_{x}^{n})+\frac{\mu }{\Delta t}
	(q^{n}+P_{h}^{0}\psi (t_{n})-(q^{n-1}+P_{h}^{0}\psi (t_{n-1})), \widehat{q}^{n})\\
	+\frac{\mu }{\Delta t}(q^{n}+P_{h}^{0}\psi (t_{n})-(q^{n-1}+P_{h}^{0}\psi (t_{n-1})), \widehat{q}^{n})=0\\ \\

	\frac{c}{2\Delta t}(\left\vert \left\vert R^{n}-R^{n-1}\right\vert
	\right\vert ^{2}+\left\vert \left\vert R^{n}\right\vert \right\vert
	^{2}-\left\vert \left\vert R^{n-1}\right\vert \right\vert ^{2})+\frac{c}{
		\Delta t}(P_{h}^{0}\theta (t_{n})-P_{h}^{0}\theta (t_{n-1}), R^{n}) \\ 	+\frac{d}{\Delta t}(P_{h}^{0}P(t_{n})-P_{h}^{0}P(t_{n-1}), R^{n})+ \frac{d}{
		\Delta t}(J^{n}-J^{n-1}, R^{n})+\kappa ((P_{h}^{0}\theta
	(t_{n}))_{x}, R_{x}^{n})+\kappa (R_{x}^{n}, R_{x}^{n})   \\
	 +\xi _{1}(P_{h}^{0}\psi _{t}(t_{n}), R_{x}^{n})+\xi _{1}(\widehat{q}
	^{n}, R_{x}^{n})=0
	\\ \\
	\frac{r}{2\Delta t}(\left\vert \left\vert J^{n}-J^{n-1}\right\vert
	\right\vert ^{2}+\left\vert \left\vert J^{n}\right\vert \right\vert
	^{2}-\left\vert \left\vert J^{n-1}\right\vert \right\vert ^{2})+\frac{r}{%
		\Delta t}(P_{h}^{0}P(t_{n})+P_{h}^{0}P(t_{n-1}), J^{n})  \\
+\frac{d}{\Delta t}(P_{h}^{0}\theta (t_{n})-P_{h}^{0}\theta
	(t_{n-1}), J^{n})+ \frac{d}{\Delta t}%
	(R^{n}-R^{n-1}, J^{n})+h((P_{h}^{0}P(t_{n}))_{x}, J_{x}^{n})+h(J_{x}^{n}, J_{x}^{n})
	 \\
+\xi _{2}(P_{h}^{0}\psi _{t}(t_{n}), J_{x}^{n})+\xi _{2}(\widehat{q}
	^{n}, J_{x}^{n})=0.
	\end{array}%
	\right.
	\end{eqnarray}	
	\textbf{Step 2: }Now let $\bar{\varphi}_{h}=\widehat{e}^{n}$, $\bar{\psi}_{h}=%
	\widehat{q}^{n},$ $\bar{\theta}_{h}=R^{n},$ $\bar{P}_{h}=J^{n}$ in (\ref{1}) and combine it with the previous system
	\begin{equation}
	\left\{
	\begin{array}{ll}
	\frac{\rho _{1}}{2\Delta t}(\left\vert \left\vert \widehat{e}^{n}-\widehat{e}
	^{n-1}\right\vert \right\vert ^{2}+\left\vert \left\vert \widehat{e}
	^{n}\right\vert \right\vert ^{2}-\left\vert \left\vert \widehat{e}
	^{n-1}\right\vert \right\vert ^{2})+\kappa (e_{x}^{n}+q^{n}, \widehat{e}_{x}^{n})
	\\
	=\rho _{1}(\varphi _{tt}-\frac{P_{h}^{0}\varphi _{t}(t_{n})-P_{h}^{0}\varphi
		_{t}(t_{n-1})}{\Delta t}, e^{n})+\kappa (\varphi _{x}+\psi
	-((P_{h}^{0}\varphi (t_{n}))_{x}+P_{h}^{0}\psi (t_{n})), \widehat{e}_{x}^{n}) \\
	\\
	\frac{\rho _{2}}{\Delta t}\frac{\rho _{1}}{2k}(\left\vert \left\vert \delta
	\widehat{e}^{n}\right\vert \right\vert ^{2}-\left\vert \left\vert \delta \widehat{e}
	^{n-1}\right\vert \right\vert ^{2}+\left\vert \left\vert \delta (\widehat{e}^{n}-
	\widehat{e}^{n-1})\right\vert \right\vert ^{2})+\frac{\rho _{2}}{2\Delta t}
	(\left\vert \left\vert \widehat{e}_{x}^{n}\right\vert \right\vert
	^{2}-\left\vert \left\vert \widehat{e}_{x}^{n-1}\right\vert \right\vert
	^{2}+\left\vert \left\vert \widehat{e}_{x}^{n}-\widehat{e}_{x}^{n-1}\right\vert
	\right\vert ^{2})\\
	+\alpha (q_{x}^{n}, \widehat{q}_{x}^{n})+\kappa (e_{x}^{n}+q^{n},
	\widehat{q}^{n})+\xi _{1}(R^{n}, \widehat{q}_{x}^{n})+\xi _{2}(J^{n}, \widehat{q}_{x}^{n})
	\\
	=\frac{\rho _{2}\rho _{1}}{k}(\varphi _{tt}-\frac{P_{h}^{0}\varphi
		_{t}(t_{n})-P_{h}^{0}\varphi _{t}(t_{n-1})}{\Delta t}, \delta ^{2}\widehat{e}
	^{n})+\rho _{2}(\varphi _{ttx}-\frac{(P_{h}^{0}\varphi
		_{t}(t_{n}))_{x}-(P_{h}^{0}\varphi _{t}(t_{n-1}))_{x}}{\Delta t}, \widehat{e}
	_{x}^{n})\\
	+\alpha (\psi _{x}-(P_{h}^{0}\psi (t_{n}))_{x}, \widehat{q}
	_{x}^{n})+\kappa (\varphi _{x}+\psi -((P_{h}^{0}\varphi
	(t_{n}))_{x}+P_{h}^{0}\psi (t_{n})), \widehat{q}^{n})+ \\
	\xi _{1}(\theta -P_{h}^{0}\theta (t_{n}), \widehat{q}_{x}^{n})+\xi
	_{2}(P-P_{h}^{0}P(t_{n}), \widehat{q}_{x}^{n}) \\
	\\
	\frac{c}{2\Delta t}(\left\vert \left\vert R^{n}-R^{n-1}\right\vert
	\right\vert ^{2}+\left\vert \left\vert R^{n}\right\vert \right\vert
	^{2}-\left\vert \left\vert R^{n-1}\right\vert \right\vert ^{2})+\frac{d}{
		\Delta t}(J^{n}-J^{n-1}, R^{n})+\kappa (R_{x}^{n}, R_{x}^{n})+\xi _{1}(\widehat{q}
	^{n}, R_{x}^{n}) \\
	=c(\theta _{t}-\frac{(P_{h}^{0}\theta (t_{n})-P_{h}^{0}\theta (t_{n-1}))}{%
		\Delta t}, R^{n})+d(P_{t}-\frac{P_{h}^{0}P(t_{n})-P_{h}^{0}P(t_{n-1})}{\Delta
		t}, R^{n})+\kappa (\theta _{x}-(P_{h}^{0}\theta (t_{n}))_{x}, R_{x}^{n})+ \\
	\xi _{1}(\psi _{t}-P_{h}^{0}\psi _{t}(t_{n}), R_{x}^{n}) \\
	\\
	\frac{r}{2\Delta t}(\left\vert \left\vert J^{n}-J^{n-1}\right\vert
	\right\vert ^{2}+\left\vert \left\vert J^{n}\right\vert \right\vert
	^{2}-\left\vert \left\vert J^{n-1}\right\vert \right\vert ^{2})+\frac{d}{
		\Delta t}(R^{n}-R^{n-1}, J^{n})+h(J_{x}^{n}, J_{x}^{n})+\xi _{2}(\widehat{q}
	^{n}, J_{x}^{n}) \\
	=r(P_{t}-\frac{P_{h}^{0}P(t_{n})-P_{h}^{0}P(t_{n-1})}{\Delta t}, J^{n})+d(\theta _{t}-\frac{(P_{h}^{0}\theta (t_{n})-P_{h}^{0}\theta
		(t_{n-1}))}{\Delta t}, J^{n})+h(P_{x}-(P_{h}^{0}P(t_{n}))_{x}, J_{x}^{n})+ \\
	\xi _{2}(\psi _{t}-P_{h}^{0}\psi _{t}(t_{n}), J_{x}^{n}).
	\end{array}
	\right.  \label{12}
	\end{equation}	
	\textbf{Step 3: } Summing equations of (\ref{12}) and let,
	\[
	Z_{n}=\frac{\rho _{1}\rho _{2}}{k}\left\vert \left\vert \delta \widehat{e}%
	^{n}\right\vert \right\vert ^{2}+\rho _{1}\left\vert \left\vert \widehat{e}%
	^{n}\right\vert \right\vert ^{2}+\rho _{2}\left\vert \left\vert \widehat{e}%
	_{x}^{n}\right\vert \right\vert ^{2}+\kappa \left\vert \left\vert
	e_{x}^{n}+q^{n}\right\vert \right\vert ^{2}+\alpha \left\vert \left\vert
	q_{x}^{n}\right\vert \right\vert ^{2}+c\left\vert \left\vert
	R^{n}\right\vert \right\vert ^{2}+r\left\vert \left\vert J^{n}\right\vert
	\right\vert ^{2}.
	\]
	By Young's inequality there is a positive constant $c$ such that

	Therefore
	$$Z_{n}+d(R^{n}, J^{n})-d(J^{n-1}, R^{n-1})<Z_{n-1}+2c\Delta t(Z_{n}+K_{n}).$$
	Then
	$$Z_{n}-Z_{n-1}+d(R^{n}, J^{n})-d(J^{n-1}, R^{n-1})<2c\Delta t(Z_{n}+K_{n}).$$
	Summing over $n$ we get	
	$$Z_{j}-Z_{0}+d(R^{j}, J^{j})-d(J^{0}, R^{0})<2c\Delta
	t\sum_{n=0}^{j}(Z_{n}+K_{n}).$$
	Since $Z_{0}=0$, we have
	$$Z_{j}<2c\Delta t\sum_{n=0}^{j}(Z_{n}+K_{n})-d(R^{j}, J^{j}).$$
	By applying Young's inequality
	$$Z_{j}<2c\Delta t\sum_{n=0}^{j}(Z_{n}+K_{n})+\frac{d\epsilon }{2}\left\vert
	\left\vert R^{j}\right\vert \right\vert ^{2}+\frac{d}{2\epsilon }\left\vert
	\left\vert J^{j}\right\vert \right\vert ^{2},$$
	where $\epsilon $ is chosen in the following way%
	\[
	\frac{2r}{d}>\epsilon >\frac{d}{2c},
	\]
	and
	$$CZ_{j}<2c\Delta t\sum_{n=1}^{j}Z_{n}+c\Delta t(\Delta t^{2}+h^{2}).$$
	Finally apply Grownwall' inequality and the proof is completed.
\end{proof}
\section{Numerical simulations}
In this section, we make 2 tests, the first test is done when the frictional damping with the vertical displacement. The second test is done when the frictional damping is done on angular rotation.
For both tests, we used the following data:
$h=0.0625, \Delta t=\frac{h	}{2}, \rho_1=\rho_2=k=\alpha=\xi_1=1, \xi_2=r=4\times 10^{-4}, h=0.03, c=1, \kappa=365, d=0.002$
Where all initial data function are taken to be $x^2(1-x)$. \\
\begin{figure}[h!]
	\centering
	\includegraphics[width=0.80\textwidth]{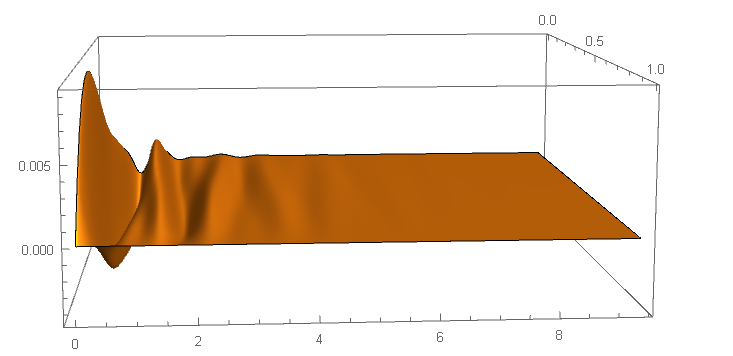}
	\caption{$\phi $ as function of $x$ and $t$.}
	\label{xip3}
\end{figure}
\begin{figure}[h!]
	\centering
	\includegraphics[width=0.80\textwidth]{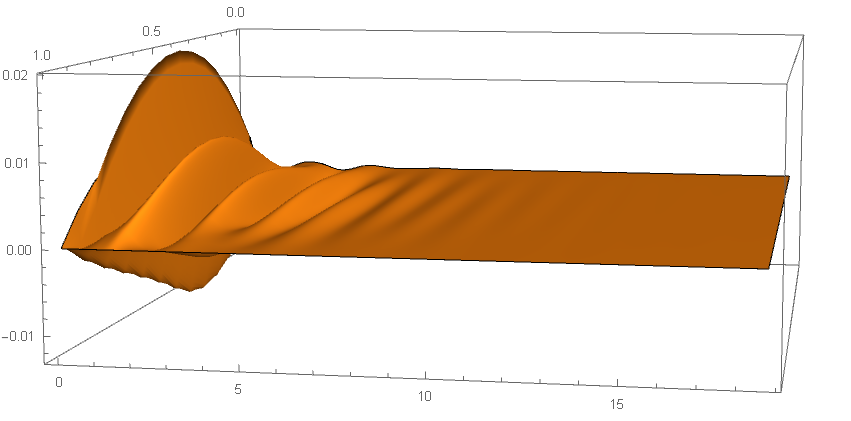}
	\caption{$\psi $ as function of $x$ and $t$.}
	\label{xip31}
\end{figure}	
\begin{figure}[h!]
	\centering
	\includegraphics[width=0.80\textwidth]{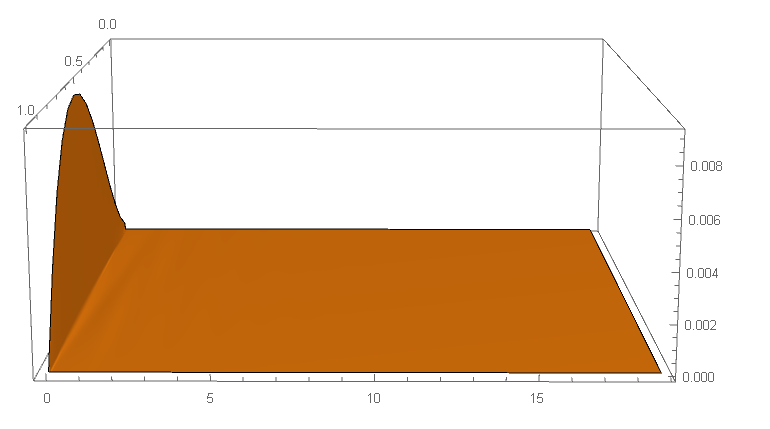}
	\caption{$\theta $ as function of $x$ and $t$.}
	\label{xip32}
\end{figure}
\begin{figure}[h!]
	\centering
	\includegraphics[width=0.80\textwidth]{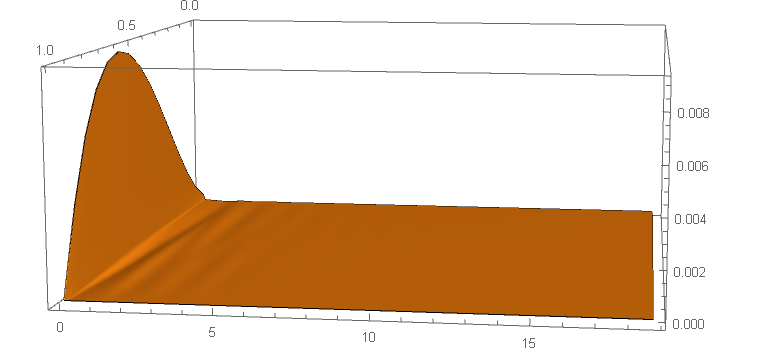}
	\caption{$p $ as function of $x$ and $t$.}
	\label{xip33}
\end{figure}
\begin{figure}[h!]
	\centering
	\includegraphics[width=0.80\textwidth]{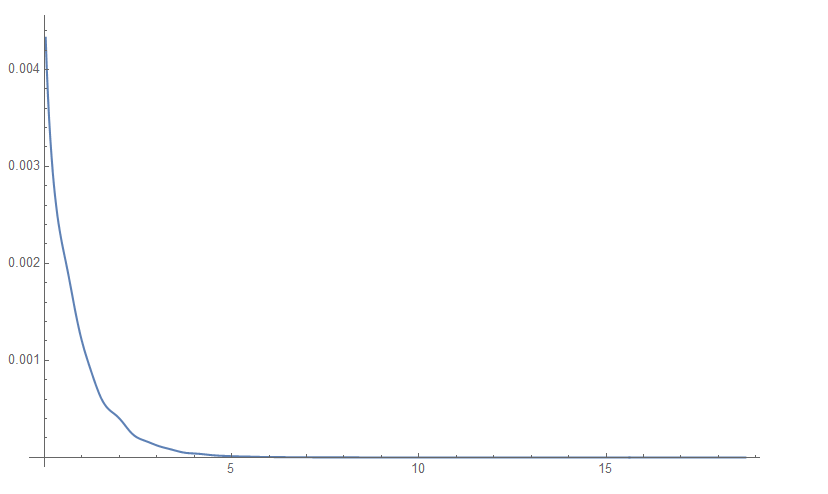}
	\caption{Energy as function of time.}
	\label{xip34}
\end{figure}
\begin{figure}[h!]
	\centering
	\includegraphics[width=0.80\textwidth]{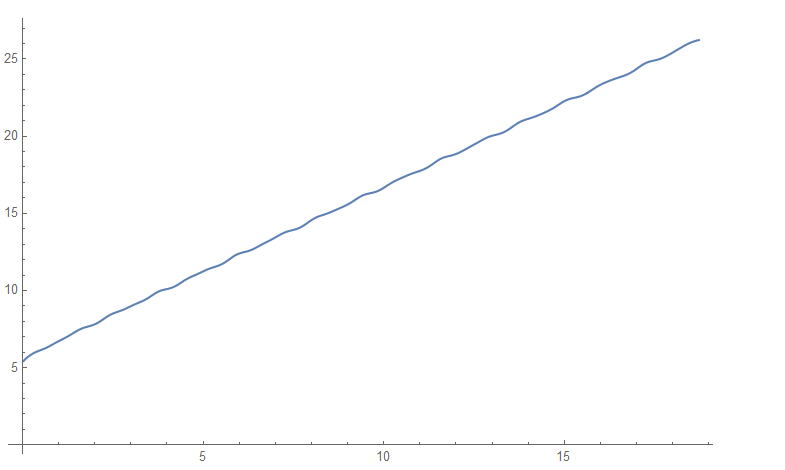}
	\caption{ $-log[energy] $ as function of time.}
	\label{xip35}
\end{figure}
\begin{figure}[h!]
	\centering
	\includegraphics[width=0.80\textwidth]{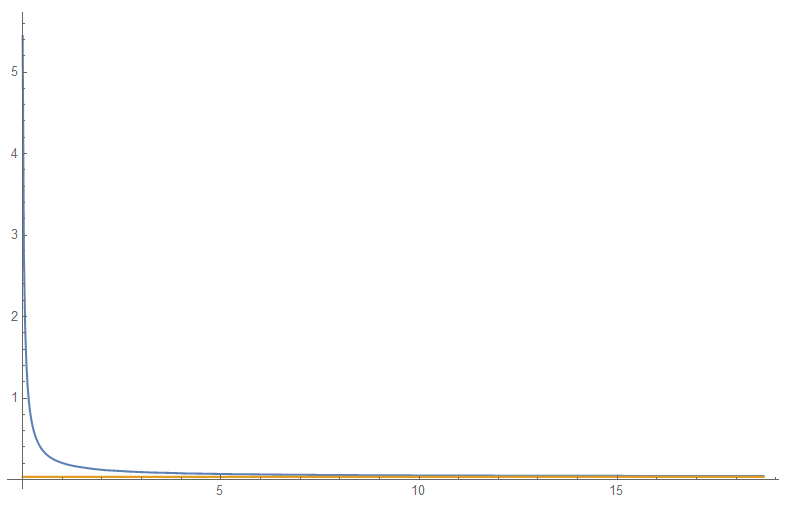}
	\caption{ Asymptotic behavior of $-log[energy]/t $ as function of time.}
	\label{xip36}
\end{figure}
\begin{figure}[h!]
	\centering
	\includegraphics[width=0.80\textwidth]{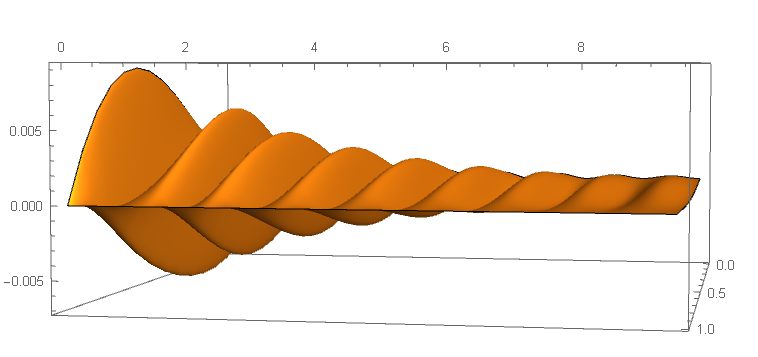}
	\caption{$\phi $ as function of $x$ and $t$.}
	\label{xip37}
\end{figure}
\begin{figure}[h!]
	\centering
	\includegraphics[width=0.80\textwidth]{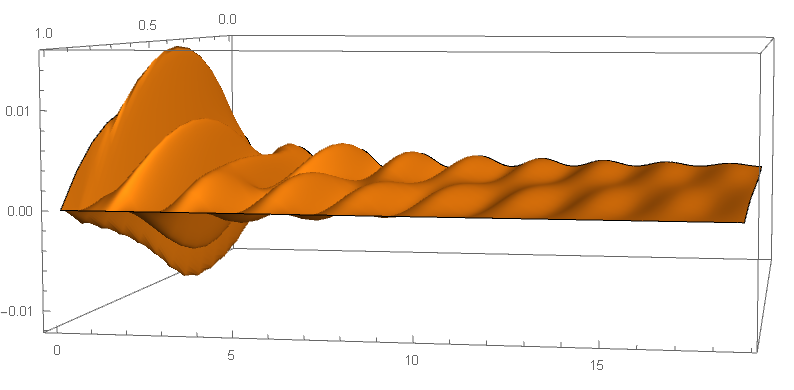}
	\caption{$\psi $ as function of $x$ and $t$.}
	\label{xip38}
\end{figure}	
\begin{figure}[h!]
	\centering
	\includegraphics[width=0.80\textwidth]{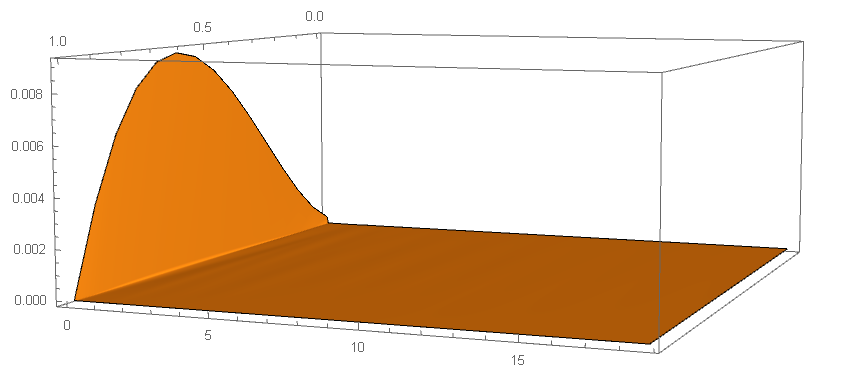}
	\caption{$\theta $ as function of $x$ and $t$.}
	\label{xip39}
\end{figure}
\begin{figure}[h!]
	\centering
	\includegraphics[width=0.80\textwidth]{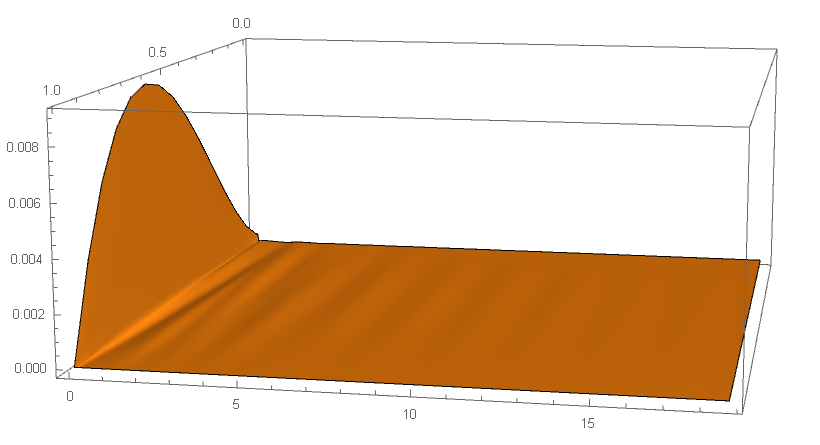}
	\caption{$p $ as function of $x$ and $t$.}
	\label{xip310}
\end{figure}
\begin{figure}[h!]
	\centering
	\includegraphics[width=0.80\textwidth]{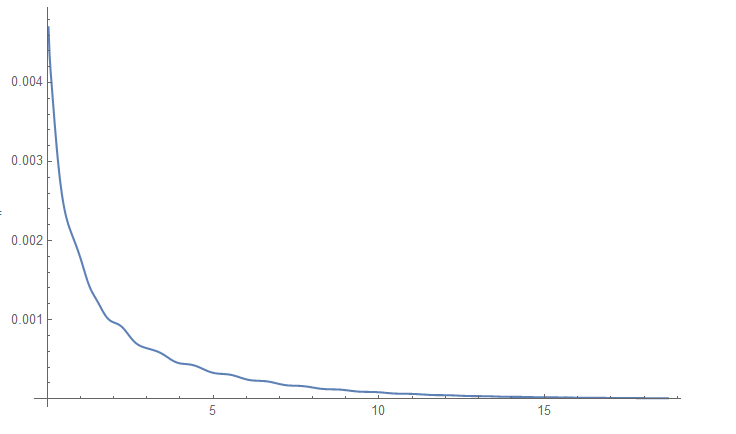}
	\caption{Energy as function of time.}
	\label{xip311}
\end{figure}
\begin{figure}[h!]
	\centering
	\includegraphics[width=0.80\textwidth]{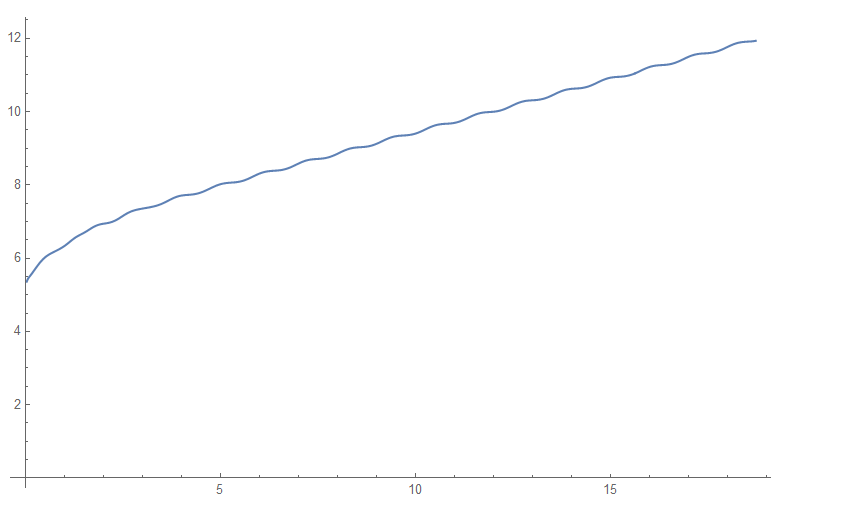}
	\caption{ $-log[energy] $ as function of time.}
	\label{xip312}
\end{figure}
\begin{figure}[h!]
	\centering
	\includegraphics[width=0.80\textwidth]{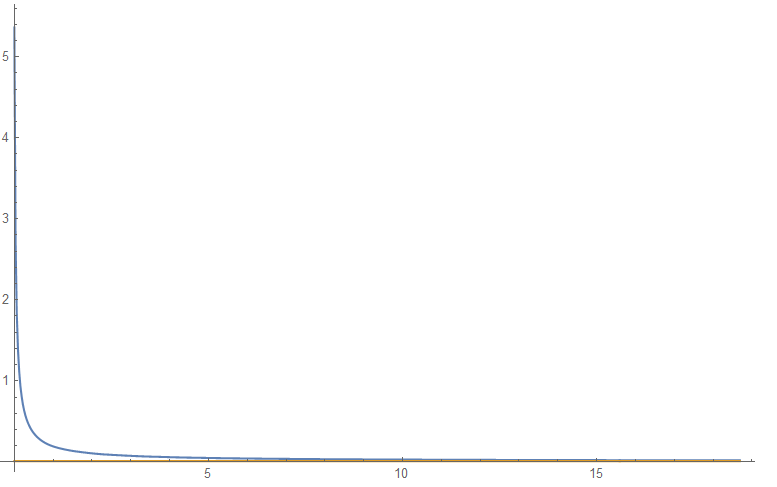}
	\caption{ Asymptotic behavior of $-log[energy]/t $ as function of time.}
	\label{xip313}
\end{figure}[h!]
\newpage
\begin{remark}
	 \begin{enumerate}  
	 	\item {\bf First test:} Figures 1-7\\
	 	 Figures 6 and 7 verify Theorem 3.10. 	 	
	 	\item {\bf Second test:} Figures 8-14 \\
 	Figures 13 and 14 verify Theorem 3.5.
	 \end{enumerate}
\end{remark} 	

\begin{remark}
	This article is a result of a joint research team that was proposed in April 2020. After the completion, correction and revision of the article, and the preparation of the final version, the authors are surprised at the publication of an article studied a similar problem \cite{Ramos2021}. We felt it necessary to highlight the points of difference between the two articles, which consist of the following:
	\begin{enumerate}
		\item  In \cite{Ramos2021}, the authors did not treated the numerical part.
		\item  The exponential stability estimate in our manuscript is taken on the frictional damping case in both angular rotation and in the vertical displacement which was not the case in \cite{Ramos2021}.
	\end{enumerate}
\end{remark}
\subsection*{Acknowledgement} For any decision, the authors would like to thank the anonymous referees and the handling editor for their careful reading and for relevant remarks/suggestions to improve the paper.

    \end{document}